\documentclass[leqno]{amsart}
\usepackage{amsmath,amsthm,amssymb,latexsym,graphicx, enumerate, mathtools, mathrsfs,cite}

\newcommand{\EE}{{\mathbb E}}
\newcommand{\NN}{{\mathbb N}}
\newcommand{\PP}{{\mathbb P}}

\newcommand{\RR}{{\mathbb R}}
\newcommand{\TT}{{\mathbb T}}

\newcommand{\abs}[1]{ \left| #1 \right|}
\newcommand{\del}{\partial}

\newcommand{\eps}{\varepsilon}
\newcommand{\Id}{\on{Id}}
\newcommand{\ind}{\mathbf{1}}
\newcommand{\ip}[1]{\langle #1 \rangle}

\newcommand{\nor}[2]{\left\|#1\right\|_{#2}}
\newcommand{\oline}[1]{\overline{#1}}
\newcommand{\oo}{\infty}
\newcommand{\pars}[1]{\left(#1\right)}

\newcommand{\uline}[1]{\underline{#1}}

\newcommand{\mcl}{\mathcal}
\newcommand{\mbb}{\mathbb}
\newcommand{\mbf}{\mathbf}
\newcommand{\on}{\operatorname}

\newtheorem{lemma}{Lemma}
\newtheorem{proposition}{Proposition}
\newtheorem{theorem}{Theorem}
\newtheorem{corollary}{Corollary}

\theoremstyle{definition}

\numberwithin{equation}{section}
\numberwithin{lemma}{section}
\numberwithin{proposition}{section}
\numberwithin{theorem}{section}
\numberwithin{corollary}{section}
\numberwithin{definition}{section}

\begin{document}
\title[Homogenization of mixing Hamilton-Jacobi equations]{Scaling limits and homogenization of mixing Hamilton-Jacobi equations}
\author{Benjamin Seeger}
\address{Place du Mar\'echal de Lattre de Tassigny, 75016 Paris, France}
\email{seeger@ceremade.dauphine.fr}

\thanks{Partially supported by the National Science Foundation Mathematical Sciences Postdoctoral Research Fellowship under Grant Number DMS-1902658}

\date{\today}

\maketitle

\begin{abstract}
	We study the homogenization of nonlinear, first-order equations with highly oscillatory mixing spatio-temporal dependence. It is shown in a variety of settings that the homogenized equations are stochastic Hamilton-Jacobi equations with deterministic, spatially homogenous Hamiltonians driven by white noise in time. The paper also contains proofs of some general regularity and path stability results for stochastic Hamilton-Jacobi equations, which are needed to prove some of the homogenization results and are of independent interest.
\end{abstract}

\section{Introduction}

The objective of this paper is to study the asymptotic behavior of Hamilton-Jacobi equations with oscillatory spatial dependence and correlated multiplicative noise dependence in time. More precisely, for small $\eps > 0$ and fixed $\gamma > 0$, we consider problems of the form
\begin{equation}\label{E:introscaling}
	\begin{dcases}
		u^\eps_t + \frac{1}{\eps^{\gamma}} \sum_{i=1}^m H^i\pars{ Du^\eps, \frac{x}{\eps},\omega} \xi^i\pars{ \frac{t}{\eps^{2\gamma}},\omega} = 0 & \text{in } \RR^d \times (0,\oo) \times \Omega \quad \text{and}\\
		u^\eps(x,0,\omega) = u_0(x) & \text{in } \RR^d \times \Omega,
	\end{dcases}
\end{equation}
where $u_0 \in UC(\RR^d)$, the space of uniformly continuous functions on $\RR^d$, $(\Omega,\mbb F, \mbb P)$ is a given probability space, $H = (H^1,H^2,\ldots, H^m): \RR^d \times \RR^d \times \Omega \to \RR^m$ possesses some sort of self-averaging property in the spatial variable, and $\xi = (\xi^1,\xi^2,\ldots,\xi^m): [0,\oo) \times \Omega \to \RR^m$ is an approximation of white noise, in the sense that
\begin{equation}\label{A:introxi}
	\left\{
	\begin{split}
	&t \mapsto \xi(t,\cdot) \text{ is piecewise continuous with $\PP$-probability one and}\\
	&\frac{1}{\delta} \xi\pars{ \frac{\cdot}{\delta^2},\cdot} \xrightarrow{\delta \to 0} dB \text{ in law, where}\\
	&B: [0,\oo) \times \Omega \to \RR \text{ is a standard Brownian motion;}
	\end{split}
	\right.
\end{equation}
see also \eqref{A:xi} below. For notational ease, when it does not cause confusion, we suppress the dependence on the parameter $\omega \in \Omega$.

We then identify a variety of settings in which \eqref{E:introscaling} approximates a stochastic partial differential equation with no spatial oscillations, that is, there exists $M \in \NN$, a deterministic $\oline{\mbf H} = (\oline{H}^1, \oline{H}^2, \ldots, \oline{H}^M) \in C(\RR^d, \RR^M)$, and a Brownian motion $\mbf B = (B^1, B^2,\ldots, B^M) : [0,\oo) \times \Omega \to \RR^M$ such that, as $\eps \to 0$, $u^\eps$ converges in distribution to the unique stochastic viscosity solution $\oline{u}$ of
\begin{equation}\label{E:introlimit}
	d \oline{u} + \oline{\mbf H}(D \oline{u}) \circ d \mbf B = 0 \quad \text{in } \RR^d \times (0,\oo) \quad \text{and} \quad \oline{u}(\cdot,0) = u_0 \quad \text{in } \RR^d.
\end{equation}
We recall some aspects of the Lions-Souganidis theory of stochastic viscosity solutions in Section \ref{S:theory} below. For more details, see also \cite{LSbook, LS1, LS2, LS4, LS3, Snotes}.


The parameter $\gamma$ in \eqref{E:introscaling} encodes the relationship between the spatial and temporal oscillations. In most of the results we prove, $\gamma > 0$ must be sufficiently small, which means that the mixing in time is mild in relation to the spatial oscillations. This is motivated by the fact that, in general, the law of the white noise approximation $\xi$ has a nontrivial effect on the homogenous Hamiltonian $\oline{H}$, and even its dimension $M$. In other words, there is no ``universal'' limit of \eqref{E:introscaling} for all fields $\xi$ that satisfy \eqref{A:introxi}.


Problems of the form \eqref{E:introscaling} arise in a variety of applications, including differential games, pathwise optimal control, and front propagation. In the latter example, we consider a family of surfaces $\{\Gamma^\eps_t)_{t \ge 0} \subset \RR^d$ evolving according to the prescribed oscillatory and fluctuating normal velocity
\[
	V^\eps = - \frac{1}{\eps^\gamma} \sum_{i=1}^m A^i\pars{ n, \frac{x}{\eps}} \xi^i \pars{ \frac{t}{\eps^{2\gamma}} ,\omega},
\]
where, for each $i = 1,2,\ldots,m$, $A^i: S^{d-1} \times \RR^d \times \Omega \to \RR$ is continuous in the first two variables, $S^{d-1} \subset \RR^d$ is the unit sphere, and $n \subset S^{d-1}$ is the outward unit normal vector to the surface $\Gamma^\eps_t$ at the point $x$. In general, such an interfacial motion develops singularities and/or discontinuities in finite time, even if all of the data is smooth. A weak sense is given to this problem with the level-set formulation (see \cite{BSS} for more details), in which $\Gamma^\eps_t$ is identified with the zero-level set of $u^\eps(\cdot,t)$, where $u^\eps$ solves \eqref{E:introscaling} with the Hamiltonians taking the form
\[
	H^i(p,x,\omega) = A^i\pars{ \frac{p}{|p|}, x, \omega} |p|.
\]
Under certain structural conditions on $A$, as $\eps \to 0$, $u^\eps$ converges locally uniformly and in distribution to the solution $\oline{u}$ of 
\begin{equation}\label{E:introeffectivefronts}
	d \oline{u} + \oline{\mbf A} \pars{ \frac{D \oline{u} }{ \abs{ D \oline{u}} } } |D\oline{u}| \circ d\mbf B = 0 \quad \text{in } \RR^d \times (0,\oo) \quad \text{and} \quad \oline{u}(\cdot,0) = u_0 \quad \text{in } \RR^d,
\end{equation}
where $\oline{\mbf A}: S^{d-1} \to \RR^M$ is deterministic and continuous and $\mbf B: [0,\oo) \times \Omega \to \RR^M$ is a Brownian motion. Through the level-set formulation, this corresponds to a collection of surfaces $(\oline{\Gamma}_t)_{t \ge 0}$ evolving according to the normal velocity 
\[
	\oline{V} = -  \oline{\mbf A}(n) \circ d\mbf B.
\]

There is an extensive literature on the approximation of stochastic partial differential equations by equations with mixing time dependence. For instance, results of this type for linear and semilinear parabolic partial differential equations were obtained by Bouc and Pardoux \cite{BP}, Kushner and Huang \cite{KH}, and Watanabe \cite{W}, and partial differential equations with spatial averaging and time fluctuations have been studied by Campillo, Kleptsyna, and Piatnitski \cite{CKP} and Pardoux and Piatnitski \cite{PP}. 

The main mathematical purpose of this work is to extend the above results to equations of first order and with nonlinear dependence on the gradient. Due to the highly oscillatory dependence in time, obtaining regularity estimates is far from straightforward, and, therefore, neither is establishing the tightness of probability measures. In addition, the nonlinear nature of the problem gives rise to further difficulties in identifying the limiting equation.

\subsection{The main results}

We now give an informal summary of the main results of the paper. Precise assumptions and statements can be found later on.


We divide the results into two cases, depending on whether $m = 1$ (the single-noise case) or $m > 1$ (the multiple-noise case).

\subsubsection{The single-noise case}

The problem of interest, for some convex and coercive $H: \RR^d \times \RR^d \times \Omega \to \RR$ and a white noise approximation $\xi: [0, \oo) \times \Omega \to \RR$, is
\begin{equation}\label{E:introsinglenoise}
	u^\eps_t + \frac{1}{\eps^\gamma} H\pars{Du^\eps,\frac{x}{\eps},\omega} \xi\pars{\frac{t}{\eps^{2\gamma}}, \omega} =0 \quad \text{in } \RR^d \times (0,\oo) \quad \text{and} \quad u^\eps(\cdot,0) = u_0 \quad \text{in } \RR^d.
\end{equation}

Many different assumptions for the dependence of the random Hamiltonian on space are covered by the results in Section \ref{S:singlenoise}. The Hamiltonian may even be allowed to depend on the ``slow'' spatial variable, as in
\[
	u^\eps_t + \frac{1}{\eps^\gamma} H\pars{Du^\eps,x,\frac{x}{\eps},\omega} \xi \pars{ \frac{t}{\eps^{2\gamma}},\omega} = 0\quad \text{in } \RR^d \times (0,\oo) \quad \text{and} \quad u^\eps(\cdot,0) = u_0 \quad \text{in } \RR^d.
\]
The field $\xi$, meanwhile, is allowed to be any reasonable approximation of white noise, or even true white noise, as for the problem
\[
	du^\eps + H\pars{ Du^\eps, \frac{x}{\eps},\omega} \circ dB = 0 \quad \text{in } \RR^d \times (0,\oo) \quad \text{and} \quad u^\eps(\cdot,0) = u_0 \quad \text{in } \RR^d,
\]
where $B: \Omega \times [0,\oo) \to \RR$ is a standard Brownian motion.

As an example of the types of results available in this setting, we assume here that
\begin{equation}\label{A:introsinglenoise}
	\left\{
	\begin{split}
		&\text{the white noise approximation $\xi$ is piecewise smooth,}\\
		&\text{$p \mapsto H(p,x,\omega)$ is convex and coercive, uniformly for $(x,\omega) \in \RR^d \times \Omega$, and}\\
		&\text{either $x \mapsto H(\cdot,x)$ is deterministic and periodic, or}\\
		&\text{$(x,\omega) \mapsto H(\cdot,x,\omega)$ is a random, stationary-ergodic field.}
	\end{split}
	\right.
\end{equation}

\begin{theorem}\label{T:introsinglenoise}
	Let $\gamma > 0$ and $u_0 \in UC(\RR^d)$, and assume that $H$ and $\xi$ satisfy \eqref{A:introsinglenoise}. Then there exists a deterministic, convex, and coercive $\oline{H}: \RR^d \to \RR$, which depends only on $H$, and a Brownian motion $B: [0,\oo) \times \Omega \to \RR$ such that, as $\eps \to 0$, the solution $u^\eps$ of \eqref{E:introsinglenoise} converges in distribution to the unique stochastic viscosity solution $\oline{u}$ of
	\begin{equation}\label{E:introsinglenoiseeq}
		d\oline{u} + \oline{H}(D\oline{u}) \circ dB =0 \quad \text{in } \RR^d \times (0,\oo) \quad \text{and} \quad \oline{u}(\cdot,0) = u_0 \quad \text{in } \RR^d.
	\end{equation}
\end{theorem}
The convergence in probability distribution in Theorem \ref{T:introsinglenoise}, and in the subsequent results below, is understood with the topology of local-uniform convergence. See Section \ref{S:theory} below for details.

Theorem \ref{T:introsinglenoise} holds without any restrictions on the positive parameter $\gamma$, or on the correlation between the random functions $H$ and $\xi$. This has to do with regularity and stability estimates for pathwise Hamilton-Jacobi equations with convex and coercive Hamiltonians. These estimates, which are of independent interest, are presented in Appendix \ref{S:pathstability}.

\subsubsection{The multiple-noise case}

We now turn to the study of the initial value problem
\begin{equation}\label{E:intromultiplepaths}
	u^\eps_t + \frac{1}{\eps^\gamma} \sum_{i=1}^m H^i \pars{ Du^\eps, \frac{x}{\eps}}  \xi^i\pars{ \frac{t}{\eps^{2\gamma}},\omega} = 0 \quad \text{in } \RR^d \times (0,\oo) \quad \text{and} \quad u^\eps(\cdot,0) = u_0 \quad \text{in } \RR^d,
\end{equation}
where $u_0 \in UC(\RR^d)$, $m > 1$, and, for each $i = 1,2,\ldots, m$, $H^i \in C(\RR^d \times \TT^d)$ and $\xi^i: [0,\oo) \times \Omega \to \RR$ is a white noise approximation.

We will show that there exist $M \in \NN$  and, for each $j = 1,2,\ldots, M$, a continuous, deterministic, effective Hamiltonian $\oline{H}^j: \RR^d \to \RR$ and a Brownian motion $B^j$ such that, as $\eps \to 0$, $u^\eps$ converges locally uniformly and in distribution to the unique stochastic viscosity solution $\oline{u}$ of
\begin{equation}\label{E:introeffectivemultipleeq}
	d\oline{u} + \sum_{j=1}^M \oline{H}^j(D \oline u) \circ dB^j = 0 \quad \text{in } \RR^d \times (0,\oo) \quad \text{and} \quad \oline{u}(\cdot,0) = u_0 \quad \text{in } \RR^d.
\end{equation}
Despite the similarity of this statement with Theorem \ref{T:introsinglenoise}, there are some fundamental differences in the nature of the problem. Most importantly, the deterministic effective Hamiltonians, and even their number $M$, depend on the particular laws of the mixing fields.

Different types of behavior can already be seen by considering the simple problem
\begin{equation}\label{E:introdifflaws}
	u^\eps_t + \frac{1}{\eps^\gamma}|u_x|\xi^1 \pars{ \frac{t}{\eps^{2\gamma}}, \omega} + \frac{1}{\eps^\gamma} f\pars{ \frac{x}{\eps}} \xi^2 \pars{ \frac{t}{\eps^{2\gamma}} , \omega}  = 0 \quad \text{in } \RR \times (0,\oo) \quad \text{and} \quad u^\eps(\cdot,0) = u_0 \quad \text{in } \RR.
\end{equation}
In particular, the law of the limiting problem depends nontrivially on the law of the white noise approximation $(\xi^1,\xi^2)$.

\begin{theorem}\label{T:difffields}
	Assume that $f \in C^{0,1}(\TT)$, $u_0 \in UC(\RR)$, and $0 < \gamma < 1/2$. Then there exist piecewise smooth white nosie approximations
	\[
		\xi = (\xi^1,\xi^2) : [0,\oo) \times \Omega \to \RR^2
		\quad \text{and} \quad
		\tilde \xi = (\tilde \xi^1,\tilde \xi^2): [0,\oo) \times \Omega \to \RR^2,
	\]
	deterministic functions $\oline{\mbf H}\in C( \RR, \RR^2)$ and $\widetilde{\oline{\mbf H}} \in C(\RR, \RR^4)$, and Brownian motions $\mbf B: [0,\oo) \times \Omega \to \RR^2$ and $\mbf{\tilde B}: [0,\oo) \times \Omega \to \RR^4$ such that, if $u^\eps$ and $\tilde u^\eps$ are the solutions of \eqref{E:introdifflaws} with respectively the fields $\xi$ and $\tilde \xi$, then
	\[
		\lim_{\eps \to 0} u^\eps = \oline{u} \quad \text{and} \quad \lim_{\eps \to 0} \tilde u^\eps = \widetilde{\oline{u}} \quad \text{in distribution},
	\]
	where $\oline{u}$ and $\widetilde{\oline{u}}$ are the unique stochastic viscosity solutions of respectively
	\[
		d \oline{u} + \oline{\mbf H}(\oline{u}_x) \circ d \mbf B = 0 \quad \text{and} \quad	d \widetilde{\oline{u}} + \widetilde{\oline{\mbf H}}(\widetilde{\oline{u}}_x) \circ d \mbf{\tilde B} = 0 \quad \text{in } \RR \times (0,\oo)
	\]
	with $\oline{u}(\cdot,0) = \widetilde{\oline{u}}(\cdot,0) = u_0$ in $\RR$. Moreover, as $C(\RR \times [0,\oo))$-valued random variables, $\oline{u}$ and $\widetilde{\oline{u}}$ have different laws for general $u_0 \in UC(\RR^d)$.
\end{theorem}
In the above result, $\xi$ and $\tilde \xi$ are certain discrete examples satisfying \eqref{A:introxi}, but with $\tilde \xi$ taking on more values than $\xi$. It is this property that results in a more complicated limiting Hamiltonian and a higher dimensional noise.

The next result involves a slight modification of \eqref{E:introdifflaws}, which nevertheless leads to still more varied limiting phenomena. In particular, the next result demonstrates that nontrivial correlation between the fields $\xi^i$ in \eqref{E:intromultiplepaths} can create ballistic behavior.

\begin{theorem}\label{T:nonconvex}
For some $V \in C(\TT)$, $F \in C(\RR)$, and independent white noise approximations $\xi^1,\xi^2:[0,\oo) \times \Omega \to \RR$, the following hold:
\begin{enumerate}[(a)]
\item There exists a deterministic $\oline{\mbf H}\in C(\RR, \RR^3)$ and a Brownian motion $\mbf B: [0,\oo) \times \Omega \to \RR^3$ such that, if $0 < \gamma < 1/6$, $u_0 \in UC(\RR)$, and $u^\eps$ solves 
\[
	u^\eps_t + \frac{1}{\eps^\gamma} F(u_x^\eps) \xi^1\pars{ \frac{t}{\eps^{2\gamma}} ,\omega} + \frac{1}{\eps^\gamma}V\pars{ \frac{x}{\eps}} \xi^2 \pars{ \frac{t}{\eps^{2\gamma} } ,\omega} = 0 \quad \text{in } \RR \times (0,\oo) \quad \text{and} \quad u^\eps(\cdot,0) = u_0 \quad \text{in } \RR,
\]
then, as $\eps \to 0$, $u^\eps$ converges in distribution to the unique stochastic viscosity solution of
\[
	d \oline{u} + \oline{ \mbf{H}}(\oline{u}_x) \circ d \mbf B = 0 \quad \text{in } \RR \times (0,\oo) \quad \text{and} \quad \oline{u}(\cdot,0) = u_0 \quad \text{in } \RR.
\]

\item There exists $p \in \RR$ and a deterministic, nonzero constant $\oline{c} \ne 0$ such that, if $0 < \gamma < 1$ and $\tilde u^\eps$ is the solution of 
\[
	\tilde u^\eps_t + \frac{1}{\eps^\gamma}  F(\tilde u_x^\eps)\xi^1 \pars{ \frac{t}{\eps^{2\gamma}} ,\omega} + \frac{1}{\eps^\gamma}V\pars{ \frac{x}{\eps}}  \xi^1 \pars{ \frac{t}{\eps^{2\gamma}},\omega} = 0 \quad \text{in } \RR \times (0,\oo) \quad \text{and} \quad \tilde u^\eps(x,0) = p \cdot x \quad \text{in } \RR,
\]
then, with probability one, for all $T > 0$,
\[
	\lim_{\eps \to 0} \sup_{(x,t) \in \RR \times [0,T]} \abs{ \eps^\gamma u^\eps(x,t) - \oline{c} t } = 0.
\]
\end{enumerate}
\end{theorem}

Note that the limiting Hamiltonian and noise in Theorem \ref{T:nonconvex}(a) are three- rather than two-dimensional. This is a consequence of the nonconvexity of $F$, which, as it turns out, causes certain non-symmetric properties of the potential $V$ to have an effect on the limiting problem, namely, increasing the dimension of both $\oline{\mbf H}$ and $\mbf B$. 

Also, if $F: \RR \to \RR$ is convex and coercive, then the hypotheses in \eqref{A:introsinglenoise} are satisfied by the Hamiltonian $H(p,x) := F(p) + V(x)$ and the field $\xi^i$. Hence, the example in Theorem \ref{T:nonconvex}(b), for which the function $F$ is necessarily non-convex, illustrates that the convexity assumption in Theorem \ref{T:introsinglenoise} is necessary in general.

Finally, we describe a result concerning the first order, level set problem
\begin{equation}\label{E:introgenerallevelset}
	u^\eps_t + \frac{1}{\eps^\gamma} A\pars{ \frac{x}{\eps}, \frac{t}{\eps^{2\gamma}} ,\omega} |Du^\eps| = 0 \quad \text{in } \RR^d \times (0,\oo) \quad \text{and} \quad u^\eps(\cdot,0) = u_0 \quad \text{in } \RR^d,
\end{equation}
where
\begin{equation} \label{A:intromultiplefrontspeeds}
	\left\{
		\begin{split}
			&A(y,t,\omega) := \sum_{i=1}^m a^i(y) \xi^i(t,\omega), \\
			&\pars{ (\xi^1,\xi^2,\ldots,\xi^m)([k,k+1),\cdot) }_{k=0}^\oo \text{ are independent and uniformly distributed over } \{-1,1\}^m, 			
			\\
			&a^i \in C^{0,1}(\TT^d) \quad \text{for all } i = 1,2,\ldots,m, \quad \text{and} \quad \sum_{i=1}^m a^i y^i \ne 0 \text{ whenever } y^i \in \{-1,1\}.
		\end{split}
	\right.
\end{equation}
A more general result, which covers Theorem \ref{T:introlevelset} below, will be proved in Section \ref{S:multiplenoise}. Once more, the nonlinear homogenization causes interactions between the various noise coefficients that increases the dimension of the noise, in this case from $m$ to $2^{m-1}$.

\begin{theorem}\label{T:introlevelset}
	Assume that $0 < \gamma < 1/6$, $u_0 \in UC(\RR^d)$, and \eqref{A:intromultiplefrontspeeds} holds. Then there exists $\oline{\mbf A} \in C\pars{S^{d-1}, \RR^{2^{m-1}} }$ and a Brownian motion $\mbf B: [0,\oo) \times \Omega \to \RR^{2^{m-1}}$ such that, as $\eps \to 0$, the solution $u^\eps$ of \eqref{E:introgenerallevelset} converges in distribution to the stochastic viscosity solution $\oline{u}$ of
	\[
	d\oline{u} + \oline{\mbf A} \pars{ \frac{D \oline{u}^\eps}{ \abs{ D \oline{u}^\eps} } } \abs{ D \oline{u}^\eps} \circ d \mbf B = 0 \quad \text{in } \RR^d \times (0,\oo) \quad \text{and} \quad \oline{u}(\cdot,0) = u_0 \quad \text{in } \RR^d.
\]
\end{theorem}

Recall that \eqref{E:introgenerallevelset} is the level set equation for a hypersurface evolving according to the normal velocity $-\eps^{-\gamma} A(x/\eps, t/\eps^{2\gamma})$, and the limiting equation in Theorem \ref{T:introlevelset} corresponds to the level-set flow with the normal velocity $d\mathscr{B}(n,t,\omega)$, where
\[
	\mathscr B(n,t,\omega) = \oline{\mbf A} (n) \cdot \mbf B(t,\omega).
\]

\subsection{Organization of the paper}

Section \ref{S:theory} contains some tools and concepts that are used throughout the paper. The results from the single-noise and multiple-noise cases are proved in respectively Sections \ref{S:singlenoise} and \ref{S:multiplenoise}. Finally, the appendix summarizes relevant aspects of the pathwise viscosity solution theory, as well as the computation of a certain effective Hamiltonian.

\subsection{Notation} Throughout, integration with respect to the probability measure $\PP$ is denoted by $\EE$. For a domain $U \in \RR^N$, $(B)UC(U)$ is the space of (bounded) uniformly continuous functions on $U$, and $C^2_b(U)$ is the space of $C^2$ functions $f$ whose Hessian $D^2 f$ is uniformly bounded. For $H: \RR^d \to \RR$, $H^*$ is the Legendre transform of $H$. Given a set $A$, the function $\mbf 1_A$ is the indicator function of $A$. For a function $f: \RR \to \RR$ and $x_0 \in \RR$, we write $f(x_0^{\pm}) := \lim_{h \to 0, h > 0} f(x_0 \pm h)$ whenever the limit exists. The identity matrix is denoted by $\Id$. The $(d-1)$-dimensional unit sphere in $\RR^d$ is $S^{d-1}$, and the $d$-dimensional torus is $\TT^d$. When $d = 1$, we write $\TT^1 = \TT$.

\section{White noise approximations and convergence in distribution}\label{S:theory}

Throughout the paper, we use certain facts about random variables converging in distribution. More details and proofs can be found in the book of Billingsley \cite{Bill}.

Given a Polish space $\mcl A$, that is, a complete and separable metric space, a sequence of Borel probability measures $(\mu_n)_{n\ge 1}$ on $\mcl A$ is said to converge weakly to $\mu$ as $n \to \oo$ if
\[
	\lim_{n \to \oo} \int_{\mcl A} f \; d\mu_n = \int_{\mcl A} f \; d\mu \quad \text{for all } f \in C_b(\mcl A).
\]
A sequence of $\mcl A$-valued random variables $(X_n)_{n\ge 1}$ (not necessarily defined on the same probability space) is said to converge in distribution to $X$ in the space $\mcl A$, as $n \to \oo$, if the sequence of probability laws of the $X_n$'s converges weakly to the probability law of $X$.

%
%
%
%
%
%
In this paper, we focus mainly on the two spaces $C(\RR^d \times [0,\oo))$ and $C([0,\oo), \RR^M)$, which are endowed with the topology of local-uniform convergence. These spaces are metrizable with the metrics
\[
	d_s(u,v) := \sum_{k=1}^\oo \max\pars{  \max_{(x,t) \in B_k \times [0,k]} \abs{ u(x,t) - v(x,t)} , 2^{-k}} \quad \text{for } u,v \in C(\RR^d \times [0,\oo))
\]
and
\[
	d_p(\eta,\zeta) := \sum_{k=1}^\oo \max\pars{  \max_{t \in [0,k]} \abs{ \eta(t) - \zeta(t)} , 2^{-k}} \quad \text{for } \eta,\zeta \in C([0,\oo), \RR^M).
\]
For the product space, we use the metric
\[
	\mathrm{d}((u,\eta), (v,\zeta)) := d_s(u,v) + d_p(\eta,\zeta) \quad \text{for } u,v \in C(\RR^d \times [0,\oo)) \text{ and } \eta,\zeta \in C([0,\oo),\RR^M).
\]
Throughout the paper, random variables that take values in these spaces and converge in distribution are said to converge ``locally uniformly and in distribution.''

We call a random field $\xi: [0,\oo) \times \Omega \to \RR$ a white noise approximation if
\begin{equation}\label{A:xi}
	\left\{
	\begin{split}
	&t \mapsto \xi(t,\cdot) \text{ is piecewise continuous with $\PP$-probability one and}\\
	&\zeta^{\delta} \xrightarrow{\delta \to 0} B \text{ in distribution in }C([0,\oo),\RR), \text{ where}\\
	&\zeta^{\delta}(t,\cdot) := \delta \int_0^{t/\delta} \xi(s,\cdot)\;ds \text{ and }
	B: [0,\oo) \times \Omega \to \RR \text{ is a standard Brownian motion.}
	\end{split}
	\right.
\end{equation}
A random field $\xi$ satisfies \eqref{A:xi} if it is centered, stationary, and sufficiently mixing. Such fields have been studied by a variety of authors in the context of stochastic ordinary differential equations with mixing coefficients, for example, Cogburn, Hersh, and Kac \cite{CH}, Khasminskii \cite{Khas}, Papanicolaou and Varadhan \cite{PV}, and Papanicolaou and Kohler \cite{KP}. 

An example of the types of fields appearing in the above works follows. We define the mixing rate $\rho: [0,\oo) \to [0,\oo)$ associated to $\xi$ by
\begin{equation}\label{mixingrate}
	\rho(t) = \sup_{s \ge 0} \sup_{A \in \mbb F_{s+t, \oo}} \sup_{ B \in \mbb F_{0,s} } \abs{ \PP(A \mid B) - \PP(A)} \quad \text{for } t \ge 0,
\end{equation}
where, for $0 \le s \le t \le \oo$, $\mbb F_{s,t} \subset \mbb F$ is the $\sigma$-algebra generated by the maps $\pars{\omega \mapsto \xi(r, \omega)}_{r \in [s,t]}$. The field $\xi$ can then be shown to satisfy \eqref{A:xi} if
\begin{equation}\label{A:specificxi}
	\left\{
	\begin{split}
	&t \mapsto \xi(t,\omega) \quad \text{is stationary,}\\
	&\PP \pars{ \sup_{t \in [0,\oo)} \abs{ \xi(t,\cdot)} \le M} = 1 \quad \text{for some $M > 0$,} \\
	&\lim_{t \to \oo} \rho(t) = 0, \quad \int_0^\oo \left[ \rho(t) \right]^{1/2} dt < \oo, \\
	&\EE[\xi(0)] = 0, \quad \text{and} \quad 2\int_0^\oo \EE \left[ \xi(0) \xi(t) \right] dt = 1.
	\end{split}
	\right.
\end{equation}

We also mention a discrete example, one which plays an important role later in the paper, given by
\begin{equation}\label{randomwalk}
	\xi(t,\omega) = \sum_{k=1}^\oo X_k(\omega) \ind_{[k-1,k)}(t) \quad \text{for } (t,\omega) \in [0,\oo) \times \Omega,
\end{equation}
where $\pars{X^i_k}_{k=1}^\oo: \Omega \to \RR$ is a collection of mutually independent and identically distributed random variables with
\[
	\EE[X_k] = 0 \quad \text{and} \quad \EE[(X_k)^2] = 1 \quad \text{for all } k = 1,2,\ldots.
\]
For such $\xi$, the path $\zeta^{\delta}$ appearing in \eqref{A:xi} is a linearly interpolated random walk, and \eqref{A:xi} follows from Donsker's invariance principle.

%
%
%

\section{The single-noise case} \label{S:singlenoise}

In this section, we prove the homogenization results stated in the introduction when there is a single white noise approximation. The results here resemble those of the author in \cite{Se}, except that the Hamiltonians need not be smooth or uniformly convex, which allows for the treatment of level-set problems that model front propagation.

\subsection{A general convergence result}
The first result we prove in this section is not directly related to homogenization, and is general enough to be applied to a variety of asymptotic problems. We give more details on such examples, including the ones stated in the introduction, at the end of this section.

For an initial datum $u_0 \in UC(\RR^d)$, paths $(\zeta^\eps)_{\eps \ge 0}: [0,\oo) \times \Omega \to \RR$ and Hamiltonians $(H^\eps)_{\eps \ge 0}: \RR^d \times \RR^d \times \Omega \to \RR$, we consider, for $\eps > 0$, the problems
\begin{equation}\label{E:singlepath}
	du^\eps + H^\eps(Du^\eps, x,\omega)\cdot d\zeta^\eps(t,\omega) = 0 \quad \text{in } \RR^d \times (0,\oo) \quad \text{and} \quad u^\eps(\cdot,0) = u_0 \quad \text{in } \RR^d
\end{equation}
and
\begin{equation}\label{E:generallimit}
	du^0 + H^0(Du^0,x,\omega) \cdot d\zeta^0(t,\omega) = 0 \quad \text{in } \RR^d \times (0,\oo) \quad \text{and} \quad u^0(\cdot,0) = u_0 \quad \text{in } \RR^d.
\end{equation}

Let $(S^\eps_\pm(t))_{\eps, t \ge 0}: (B)UC(\RR^d) \to (B)UC(\RR^d)$ denote the solution operators for
\[
	U^\eps_{\pm, t} \pm H^\eps(DU^\eps_\pm,x,\omega) = 0 \quad \text{in } \RR^d \times (0,\oo), \quad U^\eps_\pm(\cdot,0) = \phi \quad \text{in } \RR^d,
\]
that is, $U^\eps_\pm(x,t) = S^\eps_\pm(t) \phi(x)$ for $\eps \ge 0$ and $(x,t) \in \RR^d \times [0,\oo)$. 

We assume that there exists $\Omega_0 \in \mbb F$ such that $\PP(\Omega_0) = 1$ and the following hold:
\begin{equation}\label{A:paths}
	\left\{
	\begin{split}
	&\text{$\zeta^\eps(\cdot,\omega)$ is continuous for all $\eps \ge 0$ and $\omega \in \Omega_0$, and,}\\
	&\text{as $\eps \to 0$, $\zeta^\eps \to \zeta^0$ locally uniformly and in distribution;}
	\end{split}
	\right.
\end{equation}
and
\begin{equation}\label{A:Hconv}
	\left\{
	\begin{split}
	&\text{there exist $\uline{\nu}, \oline{\nu}:[0,\oo) \to [0,\oo)$ as in \eqref{A:appHbasic} such that,}\\
	&\text{for all $\eps \ge 0$ and $\omega \in \Omega_0$, }(H^\eps(\cdot,\cdot,\omega))_{\eps \ge 0} \text{ satisfies \eqref{A:appHbasic}, and, for all $L,T,\delta > 0$,}\\
	&\lim_{\eps \to 0} \PP \pars{ \sup_{ \nor{D\phi}{\oo} \le L} \max_{(x,t) \in B_T \times [0,T] } \abs{ S^\eps_\pm(t)\phi(x) - S^0_\pm(t)\phi(x)} > \delta} = 0.
	\end{split}
	\right.
\end{equation}
Because $H^\eps(\cdot,\cdot,\omega)$ satisfies the coercivity bounds \eqref{A:appHbasic} for all $\eps \ge 0$ and $\omega \in \Omega_0$, it follows from Theorem \ref{T:pathstability} that the equations \eqref{E:singlepath} and \eqref{E:generallimit} admit unique pathwise viscosity solutions by extending the solution operator to continuous paths.

\begin{theorem}\label{T:singlepath}
	Assume \eqref{A:paths} and \eqref{A:Hconv}, and let $u_0 \in UC(\RR^d)$. Then, as $\eps \to 0$, $(u^\eps,\zeta^\eps)$ converges locally uniformly and in distribution to $(u^0,\zeta^0)$.
\end{theorem}

The key idea in the proof of Theorem \ref{T:singlepath} is to compare with solutions of intermediate equations driven by more regular paths. The stability estimates of Theorem \ref{T:pathstability} allow for this strategy to be effectively carried out. 

Throughout the proofs below, we consider paths $\eta$ that satisfy
\begin{equation}\label{A:finitesignchange}
	\left\{
	\begin{split}
	&\eta: [0,\oo) \to \RR \text{ is piecewise-$C^1$ and, for any $T > 0$,}\\
	&\text{$\dot \eta$ changes sign finitely many times on $[0,T]$}.
	\end{split}
	\right.
\end{equation}

Recall that the metric $d_s$ below, defined in Section \ref{S:theory}, metrizes the space $C(\RR^d \times [0,\oo))$ with the topology of local uniform convergence.

\begin{lemma}\label{L:fixedpath}
	Assume that $v_0 \in UC(\RR^d)$, $\eta: [0,\oo) \times \Omega \to \RR$ is such that $\eta(\cdot,\omega)$ satisfies \eqref{A:finitesignchange} for all $\omega \in \Omega_0$, and $(H^\eps)_{\eps \ge 0}$ satisfies \eqref{A:Hconv}. Let $v^\eps$ and $v^0$ solve
	\begin{equation}\label{E:veq}
		\left\{
		\begin{split}
			&v^\eps_t + H^\eps(Dv^\eps, x,\omega) \dot \eta(t,\omega) = 0 \quad \mathrm{in}\; \RR^d \times (0,\oo),\\
			&v^0_t + H(Dv^0,x,\omega) \dot \eta(t,\omega)  = 0 \quad \mathrm{in}\; \RR^d \times (0,\oo), \quad \text{and}\\
			&v^\eps(\cdot,0) = v^0(\cdot,0) = v_0 \quad \mathrm{in}\; \RR^d.
		\end{split}
		\right.
	\end{equation}
	Then, for all $\delta > 0$,
	\[
		\lim_{\eps \to 0} \PP \pars{  d_s( v^\eps , v^0) > \delta } = 0.
	\]
\end{lemma}

It is necessary to use the following well-known domain-of-dependence result for viscosity solutions of Hamilton-Jacobi equations. For a proof, see the book of Lions \cite{Lbook}.

\begin{lemma}\label{L:finspeed}
	Suppose that $G: \RR^d \times \RR^d \to \RR$ is continuous, fix $L > 0$, let $U$ and $V$ be respectively a sub- and super-solution of
	\[
		U_t = G(DU,x) \quad \text{and} \quad V_t = G(DV,x) \quad \text{in } \RR^d \times (-\oo,\oo)
	\]
	such that $\max (\nor{DU}{\oo}, \nor{DV}{\oo} )\le L$, and suppose that
	\[
		\mcl L := \sup_{(p,x) \in B_L \times \RR^d} \abs{ D_p G(p,x)} < \oo.
	\]
	Then, for all $R > 0$ and $-\oo < s < t < \oo$,
	\[
		\max_{x \in B_{R - \mcl L(t-s)}} \pars{ U(x,t) - V(x,t)} \le \max_{x \in B_R} \pars{ U(x,s) - V(x,s)}.
	\]
\end{lemma}	

The strategy for the proof of Lemma \ref{L:fixedpath} is similar to one used by the author in \cite{Se}. However, the argument is more involved here, due to the randomness of both the Hamiltonian and path, and the fact that no use is made of explicit homogenization error estimates. 

\begin{proof}[Proof of Lemma \ref{L:fixedpath}]
	Observe first that, in view of the contractive property of the equations in \eqref{E:veq}, it suffices to prove the result for $v_0 \in C^{0,1}(\RR^d)$ with $\nor{Dv_0}{\oo} \le L$ for some fixed $L > 0$. Also, it is enough to prove, for any fixed $\delta > 0$ and $T > 0$, that
	\[
		\lim_{\eps \to 0} \PP \pars{ \max_{(x,t) \in B_T \times [0,T] } \abs{ v^\eps(x,t) - v^0(x,t)} > \delta } = 0.
	\]
	
	Fix $\omega \in \Omega_0$, so that there exists a partition $\{0 = t_0 < t_1 < t_2 < \cdots < t_N = T\}$ such that $\eta(\omega)$ is monotone on each interval $[t_i,t_{i+1}]$. Fix $(x,t) \in B_T \times [0,T]$, let $i$ be such that $t \in (t_i,t_{i+1}]$, and assume without loss of generality that $\eta$ is decreasing on $[t_i,t_{i+1}]$.
	
	Set $\Delta := \eta_t - \eta_{t_i}$. Because $\eta$ is monotone on $[t_i,t_{i+1}]$, 
	\[
		v^\eps(\cdot,t) = S^\eps_+(\Delta)v^\eps(\cdot,t_i) \quad \text{and} \quad v^0(\cdot,t) = S^0_+(\Delta)v^0(\cdot,t_i).
	\]
	We then write
	\[
		v^\eps(x,t) - v^0(x,t) = \pars{S^\eps_+(\Delta)v^\eps(\cdot,t_i)(x) - S^\eps_+(\Delta)v^0(\cdot,t_i)(x)}
		+ \pars{S^\eps_+(\Delta)v^0(\cdot,t_i)(x) - S^0_+(\Delta)v^0(\cdot,t_i)(x)}.
	\]
	In view of Theorem \ref{T:pathstability}, there exists a deterministic constant $C_1 > 0$ depending only on $L$ such that
	\[
		\max(\nor{Dv^\eps}{\oo}, \nor{Dv^0}{\oo}) \le C_1.
	\]
	The convexity and uniform growth of $H^\eps$ in the gradient variable then imply that, for some deterministic constant $C_2 > 0$ depending only on $L$,
	\[
		\sup_{\eps > 0} \sup_{|p| \le C_1} \sup_{x \in \RR^d} \abs{ D_p H^\eps(p,x,\omega)} \le C_2.
	\] 
	Lemma \ref{L:finspeed} then implies that, for all $x \in B_T$,
	\[
		\abs{ S^\eps_+(\Delta)v^\eps(\cdot,t_i)(x) - S^\eps_+(\Delta)v^0(\cdot,t_i)(x)} \le \max_{y \in B_{T + C_2\Delta}} \abs{ v^\eps(y,t_i) - v^0(y,t_i)},
	\]
	and so
	\begin{equation}\label{iteration}
		\abs{v^\eps(x,t) - v^0(x,t)} \le \sum_{i=0}^{N-1} \max_{(y,\tau) \in B_{R_i} \times [0, \Delta_i]} \abs{ S^\eps_{\pm}(\tau)v^0(\cdot,t_i)(y) - S^0_{\pm}(\tau)v^0(\cdot,t_i)(y)},
	\end{equation}
	where
	\[
		\Delta_i := \abs{ \eta(t_{i+1}) - \eta(t_i)} \quad \text{and} \quad R_i := T + C_2 \sum_{k=i}^{N-1} \Delta_k,
	\]
	and the subscripts $+$ and $-$ for the solution operators in \eqref{iteration} are chosen depending on whether $\eta$ is respectively decreasing or increasing on $[t_i,t_{i+1}]$.
	
	For $M > 0$, define
	\[
		A_M := \left\{ \omega \in \Omega_0 : N(\omega) \le M, \; \max_{i=0,1,2,\ldots, N-1} \Delta_i(\omega) \le M, \; R_{N(\omega)-1}(\omega) \le M \right\}.
	\]
	Then, for any $M > 0$,
	\begin{align*}
		\PP &\pars{ \max_{(x,t) \in B_T \times [0,T] } \abs{ v^\eps(x,t) - v^0(x,t)} > \delta }
		= \PP \pars{ \Omega_0 \cap \left\{ \max_{(x,t) \in B_T \times [0,T] } \abs{ v^\eps(x,t) - v^0(x,t)} > \delta \right\}}\\
		&\le \PP \pars{ \Omega_0 \backslash A_M} + \PP \pars{ A_M \cap \left\{ \sum_{i=0}^{N-1} \max_{(y,\tau) \in B_{R_i} \times [0, \Delta_i]} \abs{ S^\eps_{\pm}(\tau)v^0(\cdot,t_i)(y) - S^0_{\pm}(\tau)v^0(\cdot,t_i)(y)} > \delta  \right\}}\\
		&\le \PP \pars{ \Omega_0 \backslash A_M} + \PP \pars{  \sup_{\nor{D\phi}{\oo} \le C_1} \max_{(x,\tau) \in B_M \times [0,M]} \abs{ S^\eps_{\pm}(\tau)\phi(x) - S^0_{\pm}(\tau)\phi(x)} > \frac{\delta}{M} },
	\end{align*}
	and so, in view of \eqref{A:Hconv},
	\[
		\limsup_{\eps \to 0} \PP\pars{ \max_{(x,t) \in B_T \times [0,T] } \abs{ v^\eps(x,t) - v^0(x,t)} > \delta } \le \PP \pars{ \Omega_0 \backslash A_M}.
	\]
	Sending $M \to \oo$ yields the result.
\end{proof}

\begin{proof}[Proof of Theorem \ref{T:singlepath}]
	Appealing to the Portmanteau Theorem (see \cite{Bill}), it suffices to show that, for any open set $\mcl U\subset C(\RR^d \times [0,\oo)) \times C([0,\oo),\RR)$,
	\[
		\liminf_{\eps \to 0} \PP\pars{ (u^\eps,\zeta^\eps) \in \mcl U} \ge \PP \pars{ (u^0,\zeta^0) \in \mcl U}.
	\]
	Recall that we metrize the space $C(\RR^d \times [0,\oo)) \times C([0,\oo),\RR)$ with the metric $\mathrm{d} := d_s + d_p$ defined in Section \ref{S:theory}. For $\sigma > 0$, define the open set
	\[
		\mcl U_\sigma := \left\{ (v,\eta) \in \mcl U : \mathrm{d}((v,\eta), (w,\tau)) > \sigma \; \text{for all } (w,\tau) \notin \mcl U \right\}.
	\]
	
	As in the proof of Lemma \ref{L:fixedpath}, it suffices to take $u_0 \in C^{0,1}(\RR^d)$ with $\nor{Du_0}{\oo} \le L$ for some fixed $L > 0$.
	
	Fix $\delta > 0$, and let $\eta: [0,\oo) \times \Omega \to \RR$ be such that, for all $\omega \in \Omega_0$, $\eta(\omega)$ satisfies \eqref{A:finitesignchange} and $d_p(\zeta^0(\omega),\eta(\omega)) < \delta$. For example, $\eta$ could be a piecewise linear interpolation of $\zeta^0$ over an appropriately defined (random) partition. 
	
	Let $v^\eps$ and $v^0$ be as in the statement of Lemma \ref{L:fixedpath} with the path $\eta$. Theorems \ref{T:LScriteria} and \ref{T:pathstability} then yield a constant $C > 0$ depending only on $L$ such that, for all $\omega \in \Omega_0$,
	\[
		d_s(u^\eps(\omega),v^\eps(\omega)) \le Cd_p(\zeta^\eps(\omega),\eta(\omega)) \quad \text{and} \quad d_s(u^0(\omega),v^0(\omega)) \le C \delta.
	\]
	Lemma \ref{L:fixedpath} gives the existence of a deterministic $\eps_0 > 0$ such that, for all $\eps \in (0,\eps_0)$, $\PP(\Omega^\delta_\eps) \ge 1 - \delta$, where
	\[
		\Omega^\delta_\eps := \left\{ \omega \in \Omega_0: d_s(v^\eps(\omega),v^0(\omega)) < \delta  \right\}. 
	\]
	Then, for all $\eps \in (0,\eps_0)$,
	\begin{align*}
		\left\{ (u^\eps,\zeta^\eps) \in \mcl U \right\} \cup \pars{\Omega^\delta_\eps}^c
		&\supset \left\{ (v^\eps, \eta, \zeta^\eps) \in \mcl U_{(C+1)\delta} \times \mcl B_\delta(\eta) \right\} \cup \pars{ \Omega^\delta_\eps}^c\\
		&\supset \left\{ (v^0, \eta, \zeta^\eps) \in \mcl U_{(C+2)\delta} \times \mcl B_\delta(\eta) \right\} \cup \pars{ \Omega^\delta_\eps}^c,
	\end{align*}
	where $\mcl B_\delta(\eta) \subset C([0,\oo),\RR)$ denotes the open ball of radius $\delta$ centered at $\eta$ with respect to the metric $d_p$. 
	
	It follows that, for all $\eps \in (0,\eps_0)$,
	\[
		\PP \pars{ (u^\eps,\zeta^\eps) \in \mcl U } \ge \PP \pars{ (v^0,\eta,\zeta^\eps) \in \mcl U_{(C+2)\delta} \times \mcl B_\delta(\eta) } - \delta,
	\]
	which, together with \eqref{A:paths}, yields, after sending $\eps \to 0$,
	\[
		\liminf_{\eps \to 0} \PP \pars{ (u^\eps,\zeta^\eps) \in \mcl U } \ge \PP \pars{ (v^0,\eta,\zeta) \in \mcl U_{(C+2)\delta} \times \mcl B_\delta(\eta)} - \delta \ge \PP \pars{ (u^0,\zeta) \in \mcl U_{(2C+3)\delta} } - \delta.
	\]
	The result now follows upon sending $\delta \to 0$.
\end{proof}

\subsection{Applications of Theorem \ref{T:singlepath}}

The assumptions needed for Theorem \ref{T:singlepath}, and in particular, those for the Hamiltonians $H^\eps$, are general enough to apply to a multitude of settings. For instance, the dependence of $H^\eps$ on $x/\eps$ can be periodic, almost periodic, or stationary ergodic. All that is needed is \eqref{A:Hconv}, that is, convergence to some $H^0$ in the solution-operator sense. Here, to have a simplified presentation, we discuss only the periodic and random settings, with $H^\eps$ given as a function of $x/\eps$ and possibly $x$.

We first prove the result from the introduction concerning the initial value problem
\begin{equation}\label{E:singlenoise}
	u^\eps_t + \frac{1}{\eps^\gamma} H\pars{Du^\eps,\frac{x}{\eps},\omega} \xi\pars{\frac{t}{\eps^{2\gamma}}, \omega} =0 \quad \text{in } \RR^d \times (0,\oo) \quad \text{and} \quad u^\eps(\cdot,0) = u_0 \quad \text{in } \RR^d
\end{equation}
for a fixed $\gamma > 0$ and $u_0 \in UC(\RR^d)$, a white noise approximation $\xi: [0,\oo) \times \Omega \to \RR$ in the sense of \eqref{A:xi}, and a Hamiltonian $H: \RR^d \times \RR^d \times \Omega \to \RR$ for which
\begin{equation}\label{A:singlenoiseH}
	\left\{
	\begin{split}
		&\text{there exists $\Omega_0 \in \mbb F$ with $\PP(\Omega_0) = 1$ and deterministic $\uline{\nu},\oline{\nu}: [0,\oo) \to [0,\oo)$ as in \eqref{A:appHbasic}}\\
		&\text{such that } H(\cdot,\cdot,\omega) \text{ satisfies \eqref{A:appHbasic} uniformly over $\omega \in \Omega_0$, and}\\
		&\text{either $y \mapsto H(\cdot,y)$ is deterministic and periodic, or}\\
		&\text{$(y,\omega) \mapsto H(\cdot,y,\omega)$ is a stationary-ergodic random field.}
	\end{split}
	\right.
\end{equation}
The latter condition for $H$ means that there exists a group of transformations $\{T_y\}_{y \in \RR^d}: \Omega \to \Omega$ such that
\begin{equation}\label{A:Hstatergod}
	\left\{
	\begin{split}
		&\PP = \PP \circ T_y \text{ for all $y \in \RR^d$,} \\
		&H(p,x,T_y \omega) = H(p, x + y, \omega) \text{ for all $(p,x,y,\omega) \in \RR^{3d} \times \Omega$, and,}\\
		&\text{if } E \in \mbb F \text{ and }T_y E =  E \text{ for all } y \in \RR^d, \text{ then } \PP(E) = 1 \text{ or }  \PP(E) = 0.
	\end{split}
	\right.
\end{equation}

For $(t,\omega) \in [0,\oo) \times \Omega \to \RR$, define
\begin{equation}\label{zeta}
	\zeta^\eps(t,\omega) := \eps^\gamma \int_0^{t/\eps^{2\gamma}} \xi(s,\omega)\;ds = \frac{1}{\eps^\gamma} \int_0^t \xi\pars{ \frac{s}{\eps^{2\gamma}},\omega}ds.	
\end{equation}

\begin{corollary}\label{C:singlenoise}
	Let $\gamma > 0$ and $u_0 \in UC(\RR^d)$ and assume that $\xi$ and $H$ satisfy respectively \eqref{A:xi} and \eqref{A:singlenoiseH}. Then there exist a deterministic, convex $\oline{H}: \RR^d \to \RR$ satisfying \eqref{A:appHbasic}, which depends only on $H$, and a Brownian motion $B: [0,\oo) \times \Omega \to \RR$ such that, as $\eps \to 0$, $(u^\eps,\zeta^\eps)$ converges locally uniformly and in distribution to $(\oline{u},B)$, where $\oline{u}$ is the unique stochastic viscosity solution of
	\begin{equation}\label{E:singlenoiseeq}
		d\oline{u} + \oline{H}(D\oline{u}) \circ dB =0 \quad \text{in } \RR^d \times (0,\oo) \quad \text{and} \quad \oline{u}(\cdot,0) = u_0 \quad \text{in } \RR^d.
	\end{equation}
\end{corollary}
Note that equation \eqref{E:singlenoiseeq} is well-posed in the stochastic viscosity sense, by merit of Theorem \ref{T:LScriteria}.

The corollary is a direct consequence of Theorem \ref{T:singlepath}, with $\zeta^\eps$ defined as in \eqref{zeta} for $\eps > 0$, $\zeta^0 = B$,
\[
	H^\eps(p,x,\omega) := H \pars{ p, \frac{x}{\eps}, \omega } \quad \text{for $\eps > 0$ and } (p,x,\omega) \in \RR^d \times \RR^d \times \Omega, \quad \text{and} \quad H^0(p,x) = \oline{H}(p)
\]
for $(p,x) \in \RR^d \times \RR^d$, where $\oline{H}$ is the deterministic, convex, effective Hamiltonian in either the periodic or random homogenization settings. The convergence in distribution of $\zeta^\eps$ to the Brownian motion follows from \eqref{A:xi}. 

Meanwhile, $H$ and $\oline{H}$ satisfy \eqref{A:Hconv} in either the periodic or random homogenization settings. This is proved in the periodic setting by Lions, Papanicolaou, and Varadhan \cite{LPV} and Evans \cite{E}, and in the random setting by Souganidis \cite{S} and Rezakhanlou and Tarver \cite{RT} (see also Armstrong and Souganidis \cite{AS} for a more general result). In either setting, the uniformity of the convergence in \eqref{A:Hconv} over $\phi$ with a bounded Lipschitz constant is a consequence of the contractive property of the equations and the compact embedding of $C^{0,1}(\RR^d)$ into $C(\RR^d)$. The fact that $\oline{H}$ satisfies the bounds in \eqref{A:appHbasic} follows from standard estimates on the effective Hamiltonian.

Finally, we note that, for $H$ satisfying \eqref{A:singlenoiseH}, the limiting problems for
\[
	u^\eps_t + H\pars{Du^\eps,\frac{x}{\eps}} = 0 \quad \text{and} \quad u^\eps_t - H\pars{ Du^\eps, \frac{x}{\eps}} = 0
\]
are, respectively,
\[
	\oline{u}_t + \oline{H}\pars{D\oline{u}} = 0 \quad \text{and} \quad \oline{u}_t - \oline{H}\pars{D\oline{u}} = 0,
\]
or, more concisely,
\begin{equation}\label{consistentH}
	\oline{(-H)} = - \oline{H}.
\end{equation}
The identity \eqref{consistentH} does not hold in general. Here, it is a consequence of the convexity of $H$ in the gradient variable. For more details, see \cite{Se}.

We next consider equations with true white noise in time, that is,
\begin{equation}\label{E:singleBrowniannoise}
	du^\eps + H\pars{ Du^\eps, \frac{x}{\eps},\omega} \circ dB = 0 \quad \text{in } \RR^d \times (0,\oo) \quad \text{and} \quad u^\eps(\cdot,0) = u_0 \quad \text{in } \RR^d,
\end{equation}
where $B: \Omega \times [0,\oo) \to \RR$ is a standard Brownian motion.

\begin{corollary}\label{C:singleBrowniannoise}
	Under the same hypotheses of Corollary \ref{C:singlenoise}, as $\eps \to 0$, the solution $u^\eps$ of \eqref{E:singleBrowniannoise} converges locally uniformly and in distribution to the solution of \eqref{E:singlenoiseeq}.
\end{corollary}

The result follows from Theorem \ref{T:singlepath}, taking $(H^\eps)_{\eps \ge 0}$ as before and $\zeta^\eps = B$ for all $\eps \ge 0$.

We now mention some results concerning the initial value problems
\begin{equation}\label{E:slowspacedep}
	u^\eps_t + \frac{1}{\eps^\gamma} H\pars{ Du^\eps,\frac{x}{\eps},x,\omega} \xi \pars{ \frac{t}{\eps^{2\gamma}} ,\omega} = 0 \quad \text{in } \RR^d \times (0,\oo) \quad \text{and} \quad u^\eps(\cdot,0) = u_0 \quad \text{in } \RR^d
\end{equation}
and
\begin{equation}\label{E:Brownianslowspacedep}
	du^\eps + H\pars{ Du^\eps,\frac{x}{\eps},x,\omega}\circ dB= 0 \quad \text{in } \RR^d \times (0,\oo) \quad \text{and} \quad u^\eps(\cdot,0) = u_0 \quad \text{in } \RR^d.
\end{equation}
The following result is a consequence of Theorem \ref{T:singlepath}, as well as the homogenization results cited above from \cite{LPV,E,S,RT}, which extend also to this setting.

\begin{corollary}\label{C:slowspacedep}
	Assume that $\gamma > 0$, $u_0 \in UC(\RR^d)$, $B:[0,\oo) \times \Omega \to \RR$ is a Brownian motion, $H$ is uniformly continuous in $B_R \times \RR^d \times \RR^d$ for each $R > 0$, and there exist $\Omega_0 \in \mbb F$ and $\uline{\nu},\oline{\nu}$ as in \eqref{A:singlenoiseH} such that, for each fixed $x \in \RR^d$, $H(\cdot,\cdot,x)$ satisfies \eqref{A:singlenoiseH}. Then there exists a deterministic $\oline{H} \in C(\RR^d \times \RR^d)$ satisfying \eqref{A:appHbasic} such that the following hold:
	\begin{enumerate}[(a)]
	\item \label{T:slowspacesmoothxi} For any $\xi: [0,\oo) \times \Omega \to \RR$ satisfying \eqref{A:xi}, if $u^\eps$ is the solution of \eqref{E:slowspacedep} and $\zeta^\eps$ is as in \eqref{zeta}, then, as $\eps \to 0$, $(u^\eps,\zeta^\eps)$ converges locally uniformly and in distribution to $(\oline{u},B)$, where $\oline{u}$ is the stochastic viscosity solution $\oline{u}$ of
	\begin{equation}\label{E:slowspacelimit}
		d\oline{u} + \oline{H}(D \oline{u},x) \circ dB = 0 \quad \text{in } \RR^d \times (0,\oo) \quad \text{and} \quad \oline{u}(\cdot,0) = u_0 \quad \text{in } \RR^d.
	\end{equation}

	\item \label{T:slowspaceBM} As $\eps \to 0$, the solution $u^\eps$ of \eqref{E:Brownianslowspacedep} converges locally uniformly and in distribution to $\oline{u}$.
	\end{enumerate}
\end{corollary}

We conclude this subsection by explaining how the above results can be applied to equations of level-set type. Indeed, if, for some $a: \Omega \to C(S^{d-1} \times \RR^d\times \RR^d)$, 
\begin{equation}\label{frontspeedH}
	H(p,y,x,\omega) = a\pars{ \frac{p}{|p|},y, x,\omega}|p|,
\end{equation}
then \eqref{E:slowspacedep} and \eqref{E:Brownianslowspacedep} become level-set equations for certain first-order interfacial motions. For some $\oline{a} \in C(S^{d-1}\times \RR^d)$, the effective Hamiltonian then has the form 
\[
	\oline{H}(p,x) := \oline{a}\pars{ \frac{p}{|p|} ,x}|p| \quad \text{for } p \in \RR^d.
\]

The Hamiltonian \eqref{frontspeedH} satisfies \eqref{A:singlenoiseH} if there exist $0 < a_- < a_+$ such that, with probability one,
\[
	a_- \le a(n, x) \le a_+ \quad \text{for all } (n,x) \in S^{d-1} \times \RR^d,
 \quad \text{and} \quad
	p \mapsto a\pars{ \frac{p}{|p|}, \cdot,} |p| \quad \text{is convex}.
\]

\section{The multiple-noise case}\label{S:multiplenoise}

We now turn to the study of the initial value problem
\begin{equation}\label{E:multiplepaths}
	u^\eps_t + \frac{1}{\eps^\gamma} \sum_{i=1}^m H^i\pars{ Du^\eps, \frac{x}{\eps}} \xi^i\pars{ \frac{t}{\eps^{2\gamma}} ,\omega} = 0 \quad \text{in } \RR^d \times (0,\oo) \quad \text{and} \quad u^\eps(\cdot,0) = u_0 \quad \text{in } \RR^d.
\end{equation}
Throughout this section, we will assume that each Hamiltonian is deterministic and periodic in space, and that, for each $i = 1, 2,\ldots, m$, $\xi^i$ is a discrete mixing field satisfying \eqref{randomwalk}, that is,
\begin{equation}\label{laterrandomwalk}
	\left\{
	\begin{split}
	&\xi^i(t,\omega) = \sum_{k=1}^\oo X^i_k(\omega) \ind_{[k-1,k)}(t) \quad \text{for } (t,\omega) \in [0,\oo) \times \Omega,\\
	&\text{where } \pars{X^i_k}_{k=1}^\oo: \Omega \to \RR \text{ are independent and identically distributed with}\\
	&\EE[X^i_k] = 0 \text{ and } \EE[(X^i_k)^2] = 1.
	\end{split}
	\right.
\end{equation}
As in \eqref{zeta}, we set, for each $i = 1,2,\ldots,m$,
\begin{equation}\label{zetas}
	\zeta^{i,\eps}(t,\omega) := \frac{1}{\eps^\gamma} \int_0^{t/\eps^{2\gamma}} \xi^i(s,\omega)\;ds \quad \text{for } (t,\omega) \in [0,\oo) \times \Omega,
\end{equation}
so that, in view of Donsker's invariance principle, for some Brownian motion $B^i: [0,\oo) \times \Omega \to \RR$,
\[
	\zeta^i \xrightarrow{\eps \to 0} B^i \quad \text{in $C([0,\oo), \RR)$ in distribution.}
\]

\subsection{Difficulties} We begin with a discussion of the general strategy of proof in the multiple noise setting, and the challenges that arise.

We first make the formal assumption, one which we later justify by choosing $\gamma$ sufficiently small (see Lemma \ref{L:homogerror} below), that $u^\eps$ is closely approximated by a solution $\oline{u}^\eps$ of an equation of the form 
\begin{equation} \label{E:generaletaeq}
	\oline{u}^\eps_t + \frac{1}{\eps^\gamma} \oline{H}\pars{ D \oline{u}^\eps,\xi \pars{ \frac{t}{\eps^{2\gamma} }, \omega} }= 0 \quad \text{in } \RR^d \times (0,\oo) \quad \text{and} \quad \oline{u}^\eps(\cdot,0) = u_0 \quad \text{in } \RR^d,
\end{equation}
via the expansion
\[
	u^\eps(x,t) \approx \oline{u}^\eps(x,t) + \eps v(x/\eps,t) + \cdots
\]
for some $v: \TT^d \times [0,\oo) \to \RR$. This yields to the following equation for $v$, for fixed $p \in \RR^d$ and $\xi \in \RR^m$:
\begin{equation}\label{E:generalcellproblem}
	\sum_{i=1}^m H^i(D_y v + p,y)\xi^i = \oline{H}(p, \xi) \quad \text{in } \RR^d. 
\end{equation}
The fixed parameters $p$ and $\xi$ stand in place of respectively the gradient $D \oline{u}^\eps(x,t)$ and the mild white noise $\eps^{-\gamma}\xi(t/\eps^{2\gamma})$. 

Note that, in deriving \eqref{E:generaletaeq}, we have assumed that $\xi \mapsto \oline{H}(\cdot,\xi)$ is positively homogenous. Later, we justify this by the fact that, under sufficient conditions on the $H^i$, \eqref{E:generalcellproblem} admits periodic solutions for a unique choice of constant $\oline{H}(p,\xi)$ on the right hand side. The positive homogeneity can then be seen from multiplying both sides of \eqref{E:generalcellproblem} by a positive constant.

If $u_0(x) = p_0 \cdot x$ for some fixed $p_0 \in \RR$, then the solution of \eqref{E:generaletaeq} is given by
\[
	\oline{u}^\eps(x,t) = p_0 \cdot x - \frac{1}{\eps^{\gamma}} \int_0^t \oline{H}\pars{ p_0, \xi \pars{ \frac{s}{\eps^{2\gamma}} } }ds.
\]
Therefore, if it can be proved that
\begin{equation}\label{centeredeffective}
	\EE \left[ \oline{H}(p_0,X_0^1,X_0^2,\ldots, X_0^m) \right]= 0,
\end{equation}
it then follows that $\oline{u}^\eps$ converges locally uniformly and in distribution, as $\eps \to 0$, to $p_0 \cdot x + \sigma(p_0) B(t)$, where $B$ is a standard Brownian motion and
\[
	\sigma(p_0)^2 :=  \EE\left[ \oline{H}(p_0, X_0^1,X_0^2,\ldots, X_0^m)^2 \right].
\]
However, the nonlinear nature of the problem makes it difficult to describe the limit of $\oline{u}^\eps$ as $\eps \to 0$ for general initial data $u_0 \in UC(\RR^d)$. This distinguishes the problem from those studied in \cite{BP,KH,W}, where the equations are uniformly parabolic and semilinear.

A further complication arises from the fact that, for two $\RR^m$-valued random variables $X_0$ and $\tilde X_0$ as in \eqref{laterrandomwalk}, the identity
\begin{equation}\label{samevariance}
	\EE\left[ \oline{H}(p, X_0)^2 \right] = \EE \left[ \oline{H}(p,\tilde X_0)^2 \right] \quad \text{for all } p \in \RR^d
\end{equation}
may fail in general, which indicates that the law of the field $\xi$ in equation \eqref{E:multiplepaths} can have a nontrivial effect on the limiting equation. 

As shown above, if \eqref{samevariance} does hold, then, whenever the initial data has the form $u_0(x) = p \cdot x$ for some $p \in \RR^d$, the laws of the limiting functions depend only on $p$, and not on the laws of $X_0$ and $\tilde X_0$. However, it could still be the case that the laws of the limiting functions differ for more general initial data. 

As an indication of why this is true, consider, for $u_0 \in UC(\RR)$ and two Brownian motions $B, \tilde B: [0,\oo) \times \Omega \to \RR$, the initial value problems
\[
	\left\{
	\begin{split}
	&du - u_x \circ dB = 0, \quad d\tilde{u} -|\tilde u_x| \circ d\tilde B = 0 \quad \text{in } \RR \times (0,\oo), \quad \text{and} \\
	&u(x,0) = \tilde u(x,0) = u_0 \quad \text{in } \RR.
	\end{split}
	\right.
\]
If $u_0(x) = p  x$ for some fixed $p \in \RR$, then the solutions
\[
	u(x,t) = p  x + p B(t) \quad \text{and} \quad \tilde u(x,t) = p  x + |p| \tilde B(t)
\]
have the same law as $C(\RR \times [0,\oo))$-valued random variables. However, if $u_0(x) = |x|$, then a simple calculation yields that
\[
	u(x,t) = |x + B(t)|,
\]
while it is shown in \cite{LS2, LSbook,Snotes} that
\[
	\tilde u(x,t) = \max\left\{ |x| + \tilde B(t), \max_{0 \le s \le t} \tilde B(s) \right\}.
\]
The $C(\RR \times [0,\oo))$-valued random variables $u$ and $\tilde u$ evidently do not share the same law.

\subsection{A general class of examples} We now present a class of Hamiltonians and white noise approximations for which, given any initial data $u_0 \in UC(\RR^d)$, the limit as $\eps \to 0$ of the solution $u^\eps$ of \eqref{E:multiplepaths} can be identified as the unique stochastic viscosity solution of a certain initial value problem.

We assume that the Hamiltonians satisfy
\begin{equation}\label{A:Hs}
	\left\{
	\begin{split}
		&H^i \in C^{0,1}(\RR^d \times \TT^d), \text{ and, for each $\xi \in \{-1,1\}^m$},\\
		&p \mapsto \sum_{i=1}^m H^i(p, \cdot)\xi^i \text{ is either convex or concave and}\\
		&\lim_{|p| \to +\oo} \inf_{y \in \TT^d} \sum_{i=1}^m H^i(p, y)\xi^i = +\oo
		\quad \text{or} \quad
		\lim_{|p| \to +\oo} \sup_{y \in \TT^d} \sum_{i=1}^m H^i(p, y)\xi^i = -\oo.
	\end{split}
	\right.
\end{equation}

\begin{lemma}\label{L:generalcellproblem}
Assume \eqref{A:Hs}. Then, for all $p \in \RR^d$ and $\xi \in \{-1,1\}^m$, there exists a unique constant $\oline{H}(p,\xi) \in \RR$ such that \eqref{E:generalcellproblem} admits a periodic solution. Moreover, $p \mapsto \oline{H}(p, \cdot)$ is either convex or concave, and
\begin{equation}\label{homogenouseffective}
	\oline{H}(\cdot,\lambda \xi) = \lambda \oline{H}(\cdot,\xi) \quad \text{for all } \lambda \in \RR \text{ and } \xi \in \{-1,1\}^m.
\end{equation}
\end{lemma}

\begin{proof}
The solvability of the cell problem \eqref{E:generalcellproblem} is a direct consequence of the coercivity assumption in \eqref{A:Hs} and the results of \cite{LPV,E}, as is the convexity or concavity in the gradient variable and the homogeneity in \eqref{homogenouseffective} for $\lambda > 0$. The fact that \eqref{homogenouseffective} holds for negative $\lambda$ follows from the identity \eqref{consistentH}.
\end{proof}

The mixing fields are assumed to satisfy, for $i = 1,\ldots, m$,
\begin{equation}\label{A:checkerboards}
	\left\{
	\begin{split}
	&\xi^i(t,\omega) = \sum_{k=0}^{\oo} X^i_k(\omega) \ind_{(k,k+1)})(t) \quad \text{for } (t,\omega) \in [0,\oo) \times \Omega, \text{ where} \\
	&\pars{ X^i_k}_{i=1,2,\ldots, m, \; k = 0,1,\ldots} \quad \text{are independent Rademacher random variables}.
	\end{split}
	\right.
\end{equation}
Define
\[
	\left\{
	\begin{split}
	&\mcl A^m := \{ \mbf j = (j_1, j_2, \ldots, j_l) : j_i \in \{1, 2, \ldots, m\}, \; j_1 < j_2 < \cdots < j_l \},\\
	&\abs{ \mbf j} = \abs{ (j_1, j_2, \ldots, j_l)} := l, \quad \text{and}\\
	&\mcl A^m_o := \{ \mbf j \in \mcl A^m : \abs{ \mbf j} \text{ is odd} \},
	\end{split}
	\right.
\]
and note that $\#\mcl A^m = 2^m - 1$ and $\#\mcl A^m_o = 2^{m-1}$.

For each $\mbf j = (j_1,j_2,\ldots, j_l) \in \mcl A^m$, set
\begin{equation}\label{effectiveingredients}
	\left\{
	\begin{split}
		&\xi^{\mbf j} := \xi^{j_1} \xi^{j_2} \cdots \xi^{j_l} \quad \text{for } \xi = (\xi^1,\xi^2, \ldots, \xi^m) \in \{-1,1\}^m,\\
		&\oline{H}^{\mbf j}(p) := \frac{1}{2^m} \sum_{\xi \in \{-1,1\}^m} \oline{H}(p,\xi) \xi^{\mbf j} \quad \text{for } p \in \RR^d, \\
		&X^{\mbf j}_k(\omega) := X^{j_1}_k(\omega) X^{j_2}_k(\omega) \cdots X^{j_l}_k(\omega), \\
		&\zeta^{\mbf j}(0,\omega) := 0, \quad \dot \zeta^{\mbf j}(t,\omega) := \sum_{k=0}^{\oo} X^{\mbf j}_k(\omega) \ind_{(k,k+1)}(t), \quad \text{and}\\
		&\zeta^{\mbf j, \eps}(t,\omega) := \eps^\gamma \zeta^{\mbf j}(t/\eps^{2\gamma},\omega) \quad \text{for } (t,\omega) \in [0,\oo) \times \Omega.
	\end{split}
	\right.
\end{equation}

Observe that, for each $\mbf j \in \mcl A^m_o$, $\oline{H}^{\mbf j}$ is a difference of convex functions, and that the homogeneity property \eqref{homogenouseffective} implies that $\oline{H}^{\mbf j} = 0$ whenever $|\mbf j|$ is even.

\begin{theorem}\label{T:generalmultpaths}
	Assume that $0 < \gamma < 1/6$, $u_0 \in UC(\RR^d)$, \eqref{A:Hs}, and \eqref{A:checkerboards}, and let $u^\eps$ be the solution of \eqref{E:multiplepaths}. Then there exist $2^{m-1}$ independent Brownian motions $(B^{\mbf j})_{\mbf j \in \mcl A^m_o}$, such that, in distribution,
	\[
		\pars{ u^\eps, (\zeta^{\mbf j,\eps} )_{\mbf j \in \mcl A^m_o} } \xrightarrow{\eps \to 0} \pars{ \oline{u}, (B^{\mbf j} )_{\mbf j \in \mcl A^m_o} } \quad \text{in } C(\RR^d \times [0,\oo)) \times C\pars{[0,\oo), \RR^{2^{m-1}} },
	\]
	where $\oline{u}$ is the unique stochastic viscosity solution of
	\begin{equation}\label{E:effectivemultiple}
		d\oline{u} + \sum_{\mbf j \in \mcl A^{m}_0} \oline{H}^{\mbf j}(D \oline{u}) \circ dB^{\mbf j} = 0 \quad \text{in } \RR^d \times (0,\oo) \quad \text{and} \quad \oline{u}(\cdot,0) = u_0 \quad \text{in } \RR^d.
	\end{equation}
	If $d = 1$, or if $d = 2$ and $p \mapsto \oline{H}(p,\cdot)$ is homogenous of degree $q$ for some $q \ge 1$, then the result holds for $0 < \gamma < 1/2$.
\end{theorem}

The result relies on the fact that the $\xi^i$ take their values in $\{-1,1\}$, and functions defined on $\{-1,1\}^m$ take a very particular form.

\begin{lemma}\label{L:decomp}
	Let $f: \{-1,1\}^m \to \RR$. Then 
	\begin{equation}\label{fdecomp}
		f(\xi) = f_0 + \sum_{\mbf j \in \mcl A^m} f_{\mbf j} \xi^{\mbf j},
	\end{equation}
	where
	\[
		f_0 := \frac{1}{2^m} \sum_{\xi \in \{-1,1\}^m} f(\xi) \quad \text{and} \quad
		f_{\mbf j} := \frac{1}{2^m} \sum_{\xi \in \{-1,1\}^m} f(\xi)\xi^{\mbf j}.
	\]
	If $f$ is odd, then $f_0 = 0$, and the sum in \eqref{fdecomp} is taken over $\mbf j \in \mcl A^m_o$.
\end{lemma}

\begin{proof}
	Let $\mcl F^m$ be the $2^m$-dimensional space of real-valued functions on $\{-1,1\}^m$. The $2^m$ functions in the collection $\mcl P^m := \{1, (\xi^{\mbf j})_{\mbf j \in \mcl A^m} \}$ are linearly independent elements of $\mcl F^m$, and therefore, their span is equal to it.
	
	For $f,g \in \mcl F^m$, define the inner product
	\[
		\langle f, g \rangle_{\mcl F^m} := \frac{1}{2^m} \sum_{\xi \in \{-1,1\}^m} f(\xi)g(\xi).
	\]
	With respect to $\langle \cdot, \cdot \rangle_{\mcl F^m}$, $\mcl P^m$ becomes an orthonormal basis, so that, for any $f \in \mcl F^m$,
	\[
		f = \sum_{q \in \mcl P^m} \ip{f,q}_{\mcl F^m} q,
	\]
	which is the desired formula. The statements about odd $f$ now follow easily.
\end{proof}

As a consequence of Lemma \ref{L:decomp}, the effective Hamiltonian $\oline{H}: \RR^d \times \{-1,1\}^m$ in \eqref{E:generalcellproblem} takes the form
\[
	\oline{H}(p,\xi) := \sum_{\mbf j \in \mcl A^m_o} \oline{H}^{\mbf j}(p) \xi^{\mbf j},
\]
where the functions $(H^{\mbf j})_{\mbf j \in \mcl A^m_o}$ are defined as in \eqref{effectiveingredients}.

The proof of the following lemma is elementary and thus omitted.

\begin{lemma}\label{L:feeltheBern}
	Let $\{X^j\}_{j=1}^m$ be mutually independent and Rademacher. Then the random variables defined by
	\[
		X^{\mbf j} := X^{j_1} X^{j_2} \cdots X^{j_l} \quad \text{for } \mbf j = (j_1, j_2, \ldots, j_l) \in \mcl A^m
	\]
	 are pairwise independent and Rademacher.
\end{lemma}

Now let $\oline{u}^\eps$ be the viscosity solution of the equation
\begin{equation}\label{E:intermediatemultiple}
	\oline{u}^\eps_t + \sum_{\mbf j \in \mcl A^{m}_o} \oline{H}^{\mbf j}(D \oline{u}^\eps) \dot \zeta^{\mbf j,\eps}(t,\omega) = 0 \quad \text{in } \RR^d \times (0,\oo) \quad \text{and} \quad \oline{u}^\eps(\cdot,0) = u_0 \quad \text{in } \RR^d,
\end{equation}
where the $\oline{H}^{\mbf j}$'s and $\zeta^{\mbf j}$'s are as in \eqref{effectiveingredients}.

\begin{lemma}\label{L:homogerror}
	Assume \eqref{A:Hs} and \eqref{A:checkerboards}, and let $u^\eps$ and $\oline{u}^\eps$ be the solutions of respectively \eqref{E:multiplepaths} and \eqref{E:intermediatemultiple}. Then, for any $L > 0$, there exists $C = C_L > 0$ such that, with probability one, whenever $\nor{Du_0}{\oo} \le L$, $\eps > 0$ and $T > 0$,
	\[
		\sup_{(x,t) \in \RR^d \times [0,T]} \abs{ u^\eps(x,t) - \oline{u}^\eps(x,t)} \le C(1+T)\eps^{1/3 - 2\gamma}.
	\]
	If $d = 1$, or if $d = 2$ and $p \mapsto \oline{H}(p,\cdot)$ is homogenous of degree $q$ for some $q \ge 1$, then the exponent can be replaced with $1 - 2 \gamma$.
\end{lemma}

We do not give the full details of the proof of Lemma \ref{L:homogerror}, as it is a simpler version of Lemma \ref{L:fixedpath} (see also Lemma 5.2 from \cite{Se}). The argument follows by applying results on rates of convergence for periodic homogenization of Hamilton-Jacobi equations (which are listed below) on each of the $O(1/\eps^{2\gamma})$ intervals on which $\xi^\eps(t)$ is constant. The effective equation on each of those intervals is given by
\[
	\oline{u}^\eps_t + \oline{H}(D \oline{u}^\eps, \eps^{-\gamma} \xi(t/\eps^{2\gamma}) ) = 0,
\]
which is exactly equation \eqref{E:intermediatemultiple}.

\begin{lemma}\label{L:periodicerrorest}
	Assume that $H$ is coercive in the gradient variable, periodic in the space variable, and locally Lipschitz. Let $u^\eps$ and $\oline{u}$ be the solutions of the initial value problems
	\[
		\left\{
		\begin{split}
		&u^\eps_t + H\pars{Du^\eps,\frac{x}{\eps}} = 0 \quad \text{and} \quad \oline{u}_t + \oline{H}(D \oline{u}) = 0 \quad \text{in } \RR^d \times (0,\oo), \quad \text{and} \\
		&u^\eps(\cdot,0) = \oline{u}(\cdot,0) = u_0 \quad \text{in } \RR^d.
		\end{split}
		\right.
	\]
	
	\begin{enumerate}[(a)]
	
	\item (Capuzzo-Dolcetta, Ishii \cite{CDI}) For all $L > 0$, there exists $C = C_L > 0$ such that, if $\nor{Du_0}{\oo} \le L$, then, for all $T > 0$,
	\begin{equation}\label{rate}
		\sup_{(x,t) \in \RR^d \times [0,T]} \abs{ u^\eps(x,t) - \oline{u}(x,t)} \le C(1+T)\eps^{1/3}.
	\end{equation}
	The exponent can be improved from $1/3$ to $1$ if $u_0(x) = p \cdot x$ for some fixed $p \in \RR^d$.
	
	\item (Mitake, Tran, Yu \cite{MTY}) If, in addition, $d=1$ and $p \mapsto H(p,\cdot)$ is convex, or if $d = 2$ and $p \mapsto H(p,\cdot)$ is convex and positively homogenous of some degree $q \ge 1$, then the exponent $1/3$ in \eqref{rate} can be replaced with $1$.
	\end{enumerate}
\end{lemma}

\begin{proof}[Proof of Theorem \ref{T:generalmultpaths}]
	Because the solution operators are contractive in the initial data, it suffices to assume that $u_0 \in C^{0,1}(\RR^d)$.

	The choice of $\gamma$ and Lemma \ref{L:homogerror} imply that, with probability one,
	\[
		\lim_{\eps \to 0} d_s\pars{ u^\eps, \oline{u}^\eps} = 0,
	\]
	where $d_s$ is the metric on $C(\RR^d \times [0,\oo))$ defined in Section \ref{S:theory}.
	
	In view of Lemma \ref{L:feeltheBern}, the path
	\[
		\zeta^\eps := \pars{ \zeta^{\mbf j, \eps} }_{\mbf j \in \mcl A^m_o} \in C\pars{ [0,\oo), \RR^{2^{m-1}} }
	\]
	is a random walk which, as $\eps \to 0$, converges in distribution to a $2^{m-1}$-dimensional Brownian motion $B := \pars{B^{\mbf j}}_{\mbf j \in \mcl A^m_o}$.
	
	For the fixed initial datum $u_0 \in C^{0,1}(\RR^d)$, let
	\[
		S: C\pars{[0,\oo),\RR^{2^{m-1}}} \ni \zeta \mapsto v \in C(\RR^d \times [0,\oo))
	\]
	be the solution operator for the equation
	\[
		dv + \sum_{\mbf j \in \mcl A^{m}_0} \oline{H}^{\mbf j}(D v) \cdot d\zeta^{\mbf j} = 0 \quad \text{in } \RR^d \times (0,\oo) \quad \text{and} \quad v(\cdot,0) = u_0 \quad \text{in } \RR^d.
	\]
	The stability result in Theorem \ref{T:LScriteria} implies that $S$ is continuous, and, therefore, so is the graph map
	 \[
	 	(S,\Id): C\pars{[0,\oo),\RR^{2^{m-1}}} \ni \zeta \mapsto (v, \zeta) \in C(\RR^d \times [0,\oo)) \times C\pars{ [0,\oo), \RR^{2^{m-1}}}.
	 \]
	 The result now follows from the Mapping Theorem and Slutsky's Theorem (see \cite{Bill}). In particular, the Mapping Theorem implies that, if $\oline{u}^\eps$ is the solution of \eqref{E:intermediatemultiple}, then, as $\eps \to 0$, $(\oline{u}^\eps, \zeta^\eps)$ converges in distribution to $(\oline{u},B)$ in $C(\RR^d \times [0,\oo)) \times C\pars{[0,T],\RR^{2^{m-1}}}$. We then conclude by appealing to Slutsky's Theorem.
\end{proof}

\subsection{A one-dimensional example}\label{SS:1dexample}
For $u_0 \in C^{0,1}(\RR)$, $\xi^1,\xi^2: [0,\oo) \times \Omega \to \RR$ as in \eqref{A:checkerboards}, and $f \in C^{0,1}(\TT)$, consider the equation
\begin{equation}\label{E:onedexample}
	u^\eps_t + \frac{1}{\eps^\gamma}   |u^\eps_x| \xi^{1}\pars{ \frac{t}{\eps^{2\gamma}} ,\omega} + \frac{1}{\eps^\gamma} f\pars{ \frac{x}{\eps}} \xi^2\pars{ \frac{t}{\eps^{2\gamma} } ,\omega} = 0 \quad \text{in } \RR \times (0,\oo) \quad \text{and} \quad u^\eps(\cdot,0) = u_0 \quad \text{in } \RR.
\end{equation}
Theorem \ref{T:generalmultpaths} implies that, if $0 < \gamma < 1/2$, then, as $\eps \to 0$, $(u^\eps,\zeta^{1,\eps}, \zeta^{2,\eps})$ converges in distribution to $(\oline{u}, B^1, B^2)$, where $\zeta^{1,\eps}$ and $\zeta^{2,\eps}$ are as in \eqref{zetas}, $B^1$ and $B^2$ are independent Brownian motions, and, for some $\oline{H}^1, \oline{H}^2: \RR \to \RR$, $\oline{u}$ is the unique stochastic viscosity solution of
\begin{equation}\label{E:twopathslimit}
	d \oline{u} + \oline{H}^1(\oline{u}_x) \circ dB^1 + \oline{H}^2(\oline{u}_x) \circ dB^2 = 0 \quad \text{in } \RR \times (0,\oo) \quad \text{and} \quad \oline{u}(\cdot,0) = u_0 \quad \text{in } \RR.
\end{equation}

To compute $\oline{H}^1$ and $\oline{H}^2$, we appeal to the following lemma, whose proof is omitted (see \cite{LPV} for similar computations). Below, define $\langle V \rangle := \int_0^1 V(y)\;dy$ for any $V \in C(\TT)$.

\begin{lemma}\label{L:simplecell}
	Let $F \in C(\TT)$. Then, for any $p \in \RR$, the equation
	\begin{equation}\label{E:simplecell}
		|p + v'(y)| + F(y) = \oline{H}(p) \quad \text{in } \TT
	\end{equation}
	admits a viscosity solution $v \in C(\TT)$ if and only if
	\[
		\oline{H}(p) = \max \left\{ \max_{y \in \TT} F(y), |p| + \langle F \rangle \right\}.
	\]
\end{lemma}

Using the formulae in \eqref{effectiveingredients} and Lemma \ref{L:simplecell}, with either $f$ or $-f$ taking the place of $F$, we explicitly compute $\oline{H}^1$ and $\oline{H}^2$, splitting into two cases depending on whether
\[
	\langle f \rangle > \frac{ \max f - \max f}{2} \quad \text{or} \quad \langle f \rangle < \frac{\max f - \min f}{2},
\]
which we refer to by saying that $f$ skews respectively upwards or downwards.

If $f$ skews upwards, then $0 \le \max f - \langle f \rangle < \langle f \rangle - \min f$,
\[
	\oline{H}^1(p) =
	\begin{dcases}
		\frac{\max f - \min f}{2} & \text{if } |p| \le \max f - \langle f \rangle, \\
		\frac{1}{2} |p| + \frac{1}{2}\pars{ \langle f \rangle - \min f} & \text{if } \max f - \langle f \rangle < |p| \le \langle f \rangle - \min f, \\
		|p| & \text{if } |p| > \langle f \rangle - \min f,
	\end{dcases}
\]
and
\[
	\oline{H}^2(p) =
	\begin{dcases}
		\frac{\max f + \min f}{2} & \text{if } |p| \le \max f - \langle f \rangle, \\
		\frac{1}{2} |p| + \frac{1}{2}\pars{ \langle f \rangle + \min f} & \text{if } \max f - \langle f \rangle < |p| \le \langle f \rangle - \min f, \\
		\langle f \rangle & \text{if } |p| > \langle f \rangle - \min f.
	\end{dcases}
\]
If $f$ skews downwards, then $0 \le \langle f \rangle - \min f < \max f - \langle f \rangle$,
\[
	\oline{H}^1(p) =
	\begin{dcases}
		\frac{\max f - \min f}{2} & \text{if } |p| \le \langle f \rangle - \min f, \\
		\frac{1}{2} |p| + \frac{1}{2}\pars{ \max f - \langle f \rangle} & \text{if } \langle f \rangle - \min f < |p| \le \max f - \langle f \rangle , \\
		|p| & \text{if } |p| > \max f - \langle f \rangle,
	\end{dcases}
\]
and
\[
	\oline{H}^2(p) =
	\begin{dcases}
		\frac{\max f + \min f}{2} & \text{if } |p| \le \langle f \rangle - \min f, \\
		\frac{1}{2} |p| + \frac{1}{2}\pars{\max f -  \langle f \rangle } & \text{if } \langle f \rangle - \min f < |p| \le \max f - \langle f \rangle , \\
		\langle f \rangle & \text{if } |p| > \max f - \langle f \rangle.
	\end{dcases}
\]

\subsection{Interfacial motions}

Theorem \ref{T:generalmultpaths} can be used to prove Theorem \ref{T:introlevelset} from the Introduction, concerning the first-order, level-set problem
\begin{equation}\label{E:generallevelset}
	u^\eps_t + \frac{1}{\eps^\gamma} A\pars{ \frac{x}{\eps}, \frac{t}{\eps^{2\gamma}},\omega } |Du^\eps| = 0 \quad \text{in } \RR^d \times (0,\oo) \quad \text{and} \quad u^\eps(\cdot,0) = u_0 \quad \text{in } \RR^d,
\end{equation}
where
\begin{equation} \label{A:multiplefrontspeeds}
	\left\{
		\begin{split}
			&A(y,t,\omega) := \sum_{i=1}^m a^i(y) \xi^i(t,\omega) \quad \text{for } (y,t,\omega) \in \TT^d \times [0,\oo) \times \Omega, \\
			&\xi^i \text{ satisfies \eqref{A:checkerboards} and } a^i \in C^{0,1}(\TT^d) \quad \text{for all } i = 1,2,\ldots,m, \quad \text{and} \\
			&\sum_{k=1}^m a^k \xi^k \ne 0 \text{ on $\TT^d$ for all $\xi \in \{-1,1\}^m$.}
		\end{split}
	\right.
\end{equation}
The Hamiltonians $H^i(p,x) := a^i(x)|p|$ then satisfy \eqref{A:Hs}. In this case, the effective Hamiltonian $\oline{H}$ given by \eqref{E:generalcellproblem} is positively homogenous in the gradient variable, and, from the formula in \eqref{effectiveingredients}, so are each of the $\oline{H}^{\mbf j}$ for $\mbf j \in \mcl A^m_o$. Therefore, each $\oline{H}^{\mbf j}$ has the form
\[
	\oline{H}^{\mbf j}(p) := \oline{a}^{\mbf j} \pars{ \frac{p}{|p|} } |p| \quad \text{for some } \oline{a}^{\mbf j} : S^{d-1} \to \RR.
\] 
For some independent Brownian motions $(B^{\mbf j})_{\mbf j \in \mcl A^m_o}$, the limiting equation is then
\[
	d\oline{u} + \sum_{\mbf j \in \mcl A^m_o} \oline{a}^{\mbf j} \pars{ \frac{D \oline{u}^\eps}{ \abs{ D \oline{u}^\eps} } } \abs{ D \oline{u}^\eps} \circ dB^{\mbf j} = 0 \quad \text{in } \RR^d \times (0,\oo) \quad \text{and} \quad \oline{u}(\cdot,0) = u_0 \quad \text{in } \RR^d.
\]

\subsection{A nonconvex example} We now turn to Theorem \ref{T:nonconvex} from the introduction. The relevant objects are defined just as in the work of Luo, Tran, and Yu \cite{LTY}.

Let $F: \RR \to \RR$ be a smooth, even function such that
\begin{equation}\label{A:nonconvexF}
	\left\{
	\begin{split}
		&\text{for some $0 < \theta_3 < \theta_2 < \theta_1$,}\\
		&F(0) = 0, \; F(\theta_2) = \frac{1}{2}, \; F(\theta_1) = F(\theta_3) = \frac{1}{3}, \; \lim_{r \to \oo} F(r) = +\oo, \\
		&\text{$F$ is strictly increasing on $[0,\theta_2] \cup [\theta_1, +\oo)$ and strictly decreasing on $[\theta_2, \theta_1]$,}
	\end{split}
	\right.
\end{equation}
and, for $0 < s < 1$, define
\begin{equation}\label{A:crookedV}
	V_s(x) := 
	\begin{dcases}
		\frac{x}{s} & \text{if } 0 \le x \le s \text{ and} \\
		\frac{1-x}{1-s} & \text{if } s < x \le 1
	\end{dcases}
\end{equation}
and extend $V_s$ to be $1$-periodic on all of $\RR$. 

For $\xi^1$ and $\xi^2$ as in \eqref{A:checkerboards}, we consider the equation
\begin{equation}\label{E:nonconvex}
	u^\eps_t + \frac{1}{\eps^\gamma}  F(u^\eps_x) \xi^1 \pars{ \frac{t}{\eps^{2\gamma}} ,\omega} + \frac{1}{\eps^\gamma} V_s \pars{ \frac{x}{\eps} } \xi^2\pars{ \frac{t}{\eps^{2\gamma}} ,\omega}  = 0 \quad \text{in } \RR \times (0,\oo) \quad \text{and} \quad u^\eps(\cdot,0) = u_0 \quad \text{in } \RR.
\end{equation}
If $F$ is replaced with a convex function, then \eqref{E:nonconvex} falls within the scope of Theorem \ref{T:generalmultpaths}, and the limiting equation resembles \eqref{E:twopathslimit}. However, the nonconvexity of $F$ and the ``crooked'' structure of $V_s$ for $s \ne 1/2$ imply that the effective Hamiltonian $\oline{H}: \RR \times \{-1,1\}^2 \to \RR$ given by the cell problem
\[
	F(p + v'(y))\xi^1 + V_s(y)\xi^2 = \oline{H}(p,\xi^1,\xi^2) \quad \text{in } \RR
\]
is not fully $1$-homogenous in the $\{-1,1\}^2$-variable. As a result, in the decomposition
\[
	\oline{H}(\cdot,\xi^1,\xi^2) = \oline{H}^0 + \oline{H}^{1}\xi^1 + \oline{H}^2 \xi^2 + \oline{H}^{\{1,2\}} \xi^1 \xi^2 \quad \text{for } \xi^1,\xi^2 \in \{-1,1\}
\]
given by Lemma \ref{L:decomp}, the term $\oline{H}^{\{1,2\}}$ does not vanish. However, it is the case, as we show below, that $\oline{H}^0 = 0$, so that \eqref{E:nonconvex} does not exhibit ballistic behavior as $\eps \to 0$.

Let $\oline{H}_s$ be the effective Hamiltonian associated to the Hamiltonian
\[
	H_s(p,x) := F(p) - V_s(x).
\]
In Appendix \ref{S:effective}, we obtain an explicit formula for $\oline{H}_s$, and deduce, in particular, that $\oline{H}_s$ satisfies \eqref{A:Hdiffconv}. Moreover, as was established in \cite{LTY}, we have $\oline{H}_s \ne \oline{H}_{s'}$ unless $s = s'$.

Simple manipulations of the cell problem, properties of viscosity solutions, and the symmetry properties
\[
	V_s(1-x) = V_{1-s}(x) \quad \text{and} \quad V_s(x) = 1 - V_{1-s}(x-s) \quad \text{for all } s \in (0,1), x \in \TT
\]
lead to the identities
\[
	\left\{
	\begin{split}
		&\oline{H}(\cdot,1,1) = \oline{H}_{1-s} + 1, \\
		&\oline{H}(\cdot,1,-1) = \oline{H}_s, \\
		&\oline{H}(\cdot,-1,1) = -\oline{H}_{1-s}, \quad \text{and}\\
		&\oline{H}(\cdot,-1,-1) = -\oline{H}_s -1,
	\end{split}
	\right.
\]
and so Lemma \ref{L:decomp} gives
\[
	\left\{
	\begin{split}
	&\oline{H}^0 = 0,\\
	&\oline{H}^1 = \frac{\oline{H}_s + \oline{H}_{1-s}+ 1}{2},\\
	&\oline{H}^2 = \frac{1}{2}, \quad \text{and}\\
	&\oline{H}^{\{1,2\}} = \frac{\oline{H}_{1-s} - \oline{H}_s}{2}.
	\end{split}
	\right.
\]
A similar proof as for Theorem \ref{T:generalmultpaths} then gives the following:

\begin{theorem}
Assume $0 < \gamma < 1/6$, $u_0 \in UC(\RR)$, $F$ and $V_s$ are as in \eqref{A:nonconvexF} and \eqref{A:crookedV}, $\xi^1$ and $\xi^2$ are as in \eqref{A:checkerboards}, the paths $(\zeta^{\mbf j,\eps})_{\mbf j \in \mcl A^2}$ are defined as in \eqref{effectiveingredients}, and $u^\eps$ is the solution of \eqref{E:nonconvex}. Then, as $\eps \to 0$, $(u^\eps, (\zeta^{\mbf j,\eps})_{\mbf j \in \mcl A^2} )$ converges locally uniformly and in distribution to $(\oline{u}, (B^{\mbf j} )_{\mbf j \in \mcl A^2} )$, where $\oline{u}$ is the unique stochastic viscosity solution of
\[
	\left\{
	\begin{split}
	&d \oline{u} + \frac{\oline{H}_s(\oline{u}_x) + \oline{H}_{1-s}(\oline{u}_x)+ 1}{2} \circ dB^1 + \frac{1}{2} \circ dB^2 + \frac{\oline{H}_{1-s}(\oline{u}_x) - \oline{H}_s(\oline{u}_x)}{2} \circ dB^{\{1,2\}} = 0 \quad \text{in } \RR \times (0,\oo) \quad \text{and}\\
	&\oline{u}(\cdot,0) = u_0 \quad \text{in } \RR.
	\end{split}
	\right.
\]
\end{theorem}

To finish this discussion and the proof of Theorem \ref{T:nonconvex}, we mention that the independence of the fields $\xi^1$ and $\xi^2$ is used in the above result, in particular, through the application of Lemma \ref{L:feeltheBern}. Indeed, for a single field $\xi$ satisfying \eqref{A:checkerboards}, consider the equation
\begin{equation}\label{E:nonconvexsinglenoise}
	u^\eps_t + \frac{1}{\eps^\gamma} \pars{ F(u^\eps_x)  - V_s \pars{ \frac{x}{\eps} } } \xi\pars{ \frac{t}{\eps^{2\gamma}} } = 0 \quad \text{in } \RR \times (0,\oo) \quad \text{and} \quad u^\eps(\cdot,0) = u_0 \quad \text{in } \RR.
\end{equation}
This equation is not covered by the result in the single-noise case, due to the fact that \eqref{consistentH} fails if $s \ne 1/2$:
\[
	\oline{(-H_s)} = - \oline{H}_{1-s} \ne - \oline{H}_s.
\]
As a consequence, we have the following:

\begin{theorem}
	Assume $0 < \gamma < 1$, $F$ and $V$ are as in \eqref{A:nonconvexF} and \eqref{A:crookedV}, $\xi$ is as in \eqref{A:checkerboards}, and, for some fixed $p_0 \in \RR$, $u^\eps$ is the solution of \eqref{E:nonconvexsinglenoise} with $u_0(x) = p_0 \cdot x$. Then, with probability one, for all $T > 0$,
	\[
		\lim_{\eps \to 0} \sup_{(x,t) \in \RR \times [0,T]} \abs{ \eps^\gamma u^\eps(x,t) - \frac{ \oline{H}_{1-s}(p_0)- \oline{H}_s(p_0)}{2} t} = 0.
	\]	
\end{theorem}

\begin{proof}
The solution $\oline{u}^\eps$ of the initial value problem
\[
	\oline{u}^\eps_t + \frac{1}{\eps^\gamma} \oline{H}\pars{ \oline{u}^\eps_x, \xi \pars{ \frac{t}{\eps^{2\gamma}},\omega}, \xi \pars{ \frac{t}{\eps^{2\gamma}},\omega}} = 0 \quad \text{in } \RR \times (0,\oo) \quad \text{and} \quad \oline{u}^\eps(x,0) = p_0 \cdot x \quad \text{in } \RR
\]
takes the form
\[
	\oline{u}^\eps(x,t) = p_0 \cdot x + \eps^\gamma \int_0^{t/\eps^{2\gamma}} \oline{H}(p_0,\xi(s),\xi(s))ds.
\]
A similar argument as for Lemma \ref{L:homogerror} gives, for some constant $C > 0$,
\[
	\sup_{(x,t) \in \RR \times [0,T]} \abs{ \eps^\gamma u^\eps(x,t) - \eps^\gamma \oline{u}^\eps(x,t)} \le C(1+T)\eps^{1- \gamma}.
\]
Note that the exponent is $1 - \gamma$, rather than $1/3 - \gamma$, because of the form of the initial datum and Lemma \ref{L:periodicerrorest}(a).

Finally, the formula for $\oline{H}$ and the law of large numbers yield, with probability one,
\[
	\lim_{\eps \to 0} \sup_{(x,t) \in \RR^d \times [0,T]} \abs{ \eps^\gamma \oline{u}^\eps(x,t) - \frac{ \oline{H}_{1-s}(p_0)- \oline{H}_s(p_0)}{2} t} = 0,
\]
which establishes the result.
\end{proof}

\subsection{Dependence of the limit on the noise approximation} We return to the equation
\begin{equation}\label{E:onedexampleagain}
	u^\eps_t + \frac{1}{\eps^\gamma}  |u^\eps_x| \xi^{1}\pars{ \frac{t}{\eps^{2\gamma}},\omega } + \frac{1}{\eps^\gamma} f\pars{ \frac{x}{\eps}} \xi^2\pars{ \frac{t}{\eps^{2\gamma}} ,\omega} = 0 \quad \text{in } \RR \times (0,\oo) \quad \text{and} \quad u^\eps(\cdot,0) = u_0 \quad \text{in } \RR,
\end{equation}
but we define the white noise approximations in such a way that the limiting equation has a different law than \eqref{E:twopathslimit}, thus establishing Theorem \ref{T:difffields} from the introduction, together with the computations in subsection \ref{SS:1dexample}.

Let $(X_k, Y_k,Z_k)_{k=0}^\oo$ be a collection of independent, Rademacher random variables, let $0 < b < a$ be such that
\[
	a^2 + b^2 = 2 \quad \text{and} \quad a(\max f - \langle f \rangle) < b (\langle f \rangle - \min f),
\]
and set
\[
	X^1_k := X_k \text{ and } X^2_k := \frac{a+b}{2} Y_k + \frac{a-b}{2} Z_k.
\]
Note that $X^1_k$ and $X^2_k$ are independent for each $k$, and
\begin{equation}\label{unit}
	\EE\left[ X^i_k\right] = 0 \quad \text{and} \quad \EE \left[ X^i_k \right]^2 = 1.
\end{equation}
For $i = 1,2$, define $\zeta^i(0) = 0$ and
\[
	\dot \zeta^i(t,\omega) = \xi^i(t,\omega) := \sum_{k=0}^\oo X^i_k(\omega) \ind_{(k,k+1)}(t) \quad \text{and} \quad \zeta^{i,\eps}(t,\omega) = \eps^\gamma \zeta^i(t/\eps^{2\gamma},\omega),
\]
and, for $\mbf j \in \left\{ \{1\}, \{2\}, \{3\}, \{1,2,3\} \right\}$, define the approximating paths $\zeta^{\mbf j, \eps}(t) := \eps^\gamma \zeta^{\mbf j}(t/\eps^{2\gamma})$, where
\[
	\left\{
	\begin{split}
	&\zeta^{\{1\},\eps} := \zeta^{1,\eps}, \quad \zeta^{\{2\}}(0) = \zeta^{\{3\}}(0) = \zeta^{\{1,2,3\}}(0) := 0, \\
	&\dot \zeta^{\{2\}}(t,\omega) := \sum_{k=0}^\oo Y_k(\omega) \ind_{(k,k+1)}(t), \quad \dot \zeta^{\{3\}}(t,\omega) := \sum_{k=0}^\oo Z_k(\omega) \ind_{(k,k+1)}(t), \quad \text{and}\\
	&\dot \zeta^{\{1,2,3\}}(t,\omega) := \sum_{k=0}^\oo X_k(\omega) Y_k(\omega) Z_k(\omega) \ind_{(k,k+1)}(t).
	\end{split}
	\right.
\]
Equation \eqref{E:onedexampleagain} can then be written as
\begin{equation}\label{E:diffonedexample}
	\left\{
	\begin{split}
	&u^\eps_t +  |u^\eps_x| \dot \zeta^{\{1\},\eps}(t,\omega) + \frac{a+b}{2} f\pars{ \frac{x}{\eps}} \dot \zeta^{\{2\},\eps}(t,\omega) + \frac{a-b}{2} f \pars{ \frac{x}{\eps}} \dot \zeta^{\{3\},\eps}(t,\omega) = 0 \quad \text{in } \RR \times (0,\oo) \quad \text{and} \\
	&u^\eps(\cdot,0) = u_0 \quad \text{in } \RR.
	\end{split}
	\right.
\end{equation}
Applying Theorem \ref{T:generalmultpaths} then gives that, if $0 < \gamma < 1/2$, then, for some independent Brownian motions $B^{\mbf j}$ with $\mbf j \in \{ \{1\}, \{2\}, \{3\}, \{1,2,3\} \}$,
\[
	\pars{ u^\eps, \zeta^{\{1\},\eps}, \zeta^{\{2\},\eps}, \zeta^{\{3\},\eps}, \zeta^{\{1,2,3\},\eps} }
	\xrightarrow{\eps \to 0}
	\pars{ \oline{u}, B^{\{1\}}, B^{\{2\}}, B^{\{3\}}, B^{\{1,2,3\}}} \quad \text{locally uniformly and in distribution},
\]
where $\oline{u}$ is the stochastic viscosity solution of
\begin{equation}\label{E:fourpathlimit}
	\left\{
	\begin{split}
	&d\oline{u} + \oline{H}^{\{1\}}(\oline{u}_x) \circ dB^{\{1\}} + \oline{H}^{\{2\}}(\oline{u}_x) \circ dB^{\{2\}} + \oline{H}^{\{3\}}(\oline{u}_x) \circ dB^{\{3\}} \\
	& \qquad + \oline{H}^{\{1,2,3\}}(\oline{u}_x) \circ dB^{\{1,2,3\}} = 0 \quad \text{in } \RR \times (0,\oo) \quad \text{and} \\
	&\oline{u}(\cdot,0) = u_0 \quad \text{in } \RR.
	\end{split}
	\right.
\end{equation}
The formulae for the effective Hamiltonians are given below, and, as can be checked, the laws of the solutions of \eqref{E:twopathslimit} and \eqref{E:fourpathlimit} differ in general, even when $u_0(x) := p_0 \cdot x$ for some fixed $p_0 \in \RR^d$:
\begin{align*}
	\oline{H}^{\{1\}}(p) :=
	\begin{dcases}
		\frac{a+b}{4} (\max f - \min f) & \text{if } 0 \le |p| \le b ( \max f - \langle f \rangle), \\
		\frac{1}{4}|p| + \frac{a}{4} ( \max f - \min f) + \frac{b}{4} (\ip{f} - \min f) & \text{if } b (\max f - \ip{f}) \le |p| \le a (\max f - \ip{f}), \\
		\frac{1}{2} |p| + \frac{a+b}{4}(\ip{f} - \min f) & \text{if } a (\max f - \ip{f}) \le |p| \le b(\ip{f} - \min f),\\
		\frac{3}{4} |p| + \frac{a}{4} ( \ip{f} - \min f) & \text{if } b(\ip f - \min f) \le |p| \le a (\ip f - \min f), \\
		|p| & \text{if } |p| \ge a (\ip f - \min f),
	\end{dcases}
\end{align*}
\begin{align*}
	\oline{H}^{\{2\}}(p) :=
	\begin{dcases}
		\frac{a+b}{4} (\max f + \min f) & \text{if } 0 \le |p| \le b ( \max f - \langle f \rangle), \\
		\frac{1}{4}|p| + \frac{a}{4} ( \max f + \min f) + \frac{b}{4} (\ip{f} + \min f) & \text{if } b (\max f - \ip{f}) \le |p| \le a (\max f - \ip{f}), \\
		\frac{1}{2} |p| + \frac{a+b}{4}(\ip{f} + \min f) & \text{if } a (\max f - \ip{f}) \le |p| \le b(\ip{f} - \min f),\\
		\frac{1}{4} |p| + \frac{a}{4} ( \ip{f} + \min f) + \frac{b}{2} \ip{f} & \text{if } b(\ip f - \min f) \le |p| \le a (\ip f - \min f), \\
		\frac{a+b}{2} \ip{f} & \text{if } |p| \ge a (\ip f - \min f),
	\end{dcases}
\end{align*}
\begin{align*}
	 \oline{H}^{\{3\}}(p) :=
	\begin{dcases}
		\frac{a-b}{4} (\max f + \min f) & \text{if } 0 \le |p| \le b ( \max f - \langle f \rangle), \\
		-\frac{1}{4}|p| + \frac{a}{4} ( \max f + \min f) - \frac{b}{4} (\ip{f} + \min f) & \text{if } b (\max f - \ip{f}) \le |p| \le a (\max f - \ip{f}), \\
		\frac{a-b}{4}(\ip{f} + \min f) & \text{if } a (\max f - \ip{f}) \le |p| \le b(\ip{f} - \min f),\\
		\frac{1}{4} |p| + \frac{a}{4} ( \ip{f} + \min f) - \frac{b}{2} \ip{f} & \text{if } b(\ip f - \min f) \le |p| \le a (\ip f - \min f), \\
		\frac{a-b}{2} \ip{f} & \text{if } |p| \ge a (\ip f - \min f),
	\end{dcases}
\end{align*}
and
\begin{align*}
	\oline{H}^{\{1,2,3\}}(p) :=
	\begin{dcases}
		\frac{a-b}{4} (\max f - \min f) & \text{if } 0 \le |p| \le b ( \max f - \langle f \rangle), \\
		-\frac{1}{4}|p| + \frac{a}{4} ( \max f - \min f) - \frac{b}{4} (\ip{f} - \min f) & \text{if } b (\max f - \ip{f}) \le |p| \le a (\max f - \ip{f}), \\
		\frac{a-b}{4}(\ip{f} - \min f) & \text{if } a (\max f - \ip{f}) \le |p| \le b(\ip{f} - \min f),\\
		-\frac{1}{4} |p| + \frac{a}{4} ( \ip{f} + \min f) & \text{if } b(\ip f - \min f) \le |p| \le a (\ip f - \min f), \\
		0 & \text{if } |p| \ge a (\ip f - \min f).
	\end{dcases}
\end{align*}

\appendix

\section{Pathwise Hamilton-Jacobi equations} \label{S:pathstability}

We give a brief overview of some facts that are needed in this paper regarding pathwise, or stochastic, viscosity solutions of the initial value problems
\begin{equation} \label{E:pathwise}
	du = H(Du,x) \cdot d\zeta \quad \text{in } \RR^d \times (0,\oo) \quad \text{and} \quad u(\cdot,0) = u_0 \quad \text{in } \RR^d
\end{equation}
and
\begin{equation}\label{E:nospace}
	du = \sum_{i=1}^m H^i(Du) \cdot d\zeta^i \quad \text{in } \RR^d \times (0,\oo) \quad \text{and} \quad u(\cdot,0) = u_0 \quad \text{in } \RR^d,
\end{equation}
where $H \in C(\RR^d \times \RR^d)$, $H^1, H^2, \ldots, H^m \in C(\RR^d)$, $\zeta, \zeta^1, \zeta^2, \ldots, \zeta^m \in C([0,\oo),\RR)$, and $u_0 \in UC(\RR^d)$. For more details, including the definitions of stochastic viscosity sub- and super-solutions and proofs of well-posedness, see \cite{LSbook, LS1, LS2, LS4, LS3, FGLS, Se,SeP,Snotes}.

Both problems \eqref{E:pathwise} and \eqref{E:nospace} fall under the scope of the classical viscosity solution theory if the driving paths are continuously differentiable, or, more generally, have finite total variation. See \cite{CL} for details on the former and \cite{I,LP} for the latter. The theory of pathwise viscosity solutions was developed by Lions and Souganidis \cite{LS2, LS4, Snotes} to study equations like \eqref{E:pathwise} and \eqref{E:nospace} when the driving paths are merely continuous.

The pathwise viscosity solution of \eqref{E:pathwise} or \eqref{E:nospace} may be identified by extending the solution operator for the equation from smooth to continuous paths. More precisely, for a fixed $u_0 \in UC(\RR^d)$, let $S_{u_0} : C^1([0,\oo)) \to C(\RR^d \times (0,\oo))$ denote either the solution operator for \eqref{E:pathwise} or \eqref{E:nospace}, both of which, under certain structural conditions on the Hamiltonians, are well-defined with the classical viscosity solution theory. 

We then say that \eqref{E:pathwise} or \eqref{E:nospace} has a unique extension to continuous paths if
\begin{equation} \label{extension}
	\left\{
	\begin{split}
	&\text{$S_{u_0}: C^1([0,\oo)) \to C(\RR^d \times [0,\oo))$ extends continuously}\\
	&\text{to $C([0,\oo))$ for any $u_0 \in UC(\RR^d)$.}
	\end{split}
	\right.
\end{equation}

As in the classical viscosity theory, there is also a notion of continuous stochastic viscosity solutions that is defined using semi-continuous sub- and super-solutions, for which a comparison principle has been proved in a variety of settings. The existence of the unique solution can then be proved alternatively through Perron's method, as by the author in \cite{SeP}. The notions of pathwise sub- and super-solutions are not used in this work, so we do not focus on them in this section. In view of the stability properties of pathwise stochastic viscosity solutions, it is always the case that the solution of \eqref{E:pathwise} or \eqref{E:nospace} obtained by extending the solution operator is a pathwise viscosity sub- and super-solution.

There is a wide class of Hamiltonians for which the spatially homogenous equation \eqref{E:nospace} is well-posed, as was shown by Lions and Souganidis in \cite{LS2}. In fact, the equation is well-posed if and only if each Hamiltonian is a difference of convex functions. In the context of the homogenization results in the body of this paper, this is important because the effective Hamiltonians need not be smooth in general.

\begin{theorem}[Lions, Souganidis \cite{LS2}] \label{T:LScriteria}
	The solution operator for \eqref{E:nospace} extends continuously in the sense of \eqref{extension} if and only if each Hamiltonian $H^i$ satisfies
	\begin{equation}\label{A:Hdiffconv}
		H = H_1 - H_2 \quad \text{for some convex } H_1, H_2: \RR^d \to \RR.
	\end{equation}
	Moreover, given $L > 0$, there exists $C = C_L > 0$ such that, for all $u_0 \in C^{0,1}(\RR^d)$ with $\nor{Du_0}{\oo} \le L$ and $\zeta_1,\zeta_2 \in C([0,\oo), \RR^m)$, if $S_{u_0}: C([0,\oo),\RR^m) \to C(\RR^d \times [0,\oo))$ is the solution operator for \eqref{E:nospace}, then
	\[
		\sup_{(x,t) \in \RR^d \times [0,T]} \abs{ S_{u_0}(\zeta_1)(x,t) - S_{u_0}(\zeta_2)(x,t)} \le C \max_{t \in [0,T]} \abs{ \zeta_1(t) - \zeta_2(t)}.
	\]
\end{theorem}

The nontrivial spatial dependence in \eqref{E:pathwise} makes the question of well-posedness more complicated. It has been proved for certain classes of Hamiltonians (see \cite{Snotes, LSbook, FGLS, Se}). We prove here a quantitative form of \eqref{extension} under less stringent regularity and structural requirements, as long as the Hamiltonian is convex and has uniform growth in the gradient variable:
\begin{equation}\label{A:appHbasic}
	\left\{
	\begin{split}
		&H \in C(\RR^d \times \RR^d), \; p \mapsto H(p,x) \text{ is convex for all $x \in \RR^d$, and}\\
		&\text{there exist convex, increasing functions } \uline{\nu}, \oline{\nu}: [0,\oo) \to \RR \text{ such that}\\
		&\uline{\nu}(|p|) \le H(p,x) \le \oline{\nu}(|p|) \quad \text{for all $(p,x) \in \RR^d \times \RR^d$.}
	\end{split}
	\right.
\end{equation}
For two smooth (or piecewise smooth) paths $\zeta^1,\zeta^2: [0,\oo) \to \RR$ and $u_0^1, u_0^2 \in C^{0,1}(\RR^d)$, consider the viscosity solutions $u^1$ and $u^2$ of
\begin{equation}
	u^j_t = H(Du^j,x) \dot \zeta^j \quad \text{in } \RR^d \times (0,\oo) \quad \text{and} \quad u^j(\cdot,0) = u_0^j \quad \text{in } \RR^d.
\end{equation}

\begin{theorem}\label{T:pathstability}
	Set $L := \max \pars{\nor{Du_0^1}{\oo}, \nor{Du_0^2}{\oo} }$. Then, for all $t > 0$ and for $j = 1, 2$,
	\[
		\nor{Du^j(\cdot,t)}{\oo} \le \uline{\nu}^{-1}\pars{ \oline{\nu}(L) },
	\]
	and, for all $T > 0$,
	\begin{align*}
		\max_{(x,t) \in \RR^d \times [0,T]} \abs{ u^1(x,t) - u^2(x,t)} &\le \max_{x \in \RR^d} \abs{ u_0^1(x) - u_0^2(x)} + \oline{\nu}(L) \max_{t \in [0,T]} \abs{ \zeta^1(t) - \zeta^2(t)} \\
		&+  \uline{\nu}(0)_- \pars{ \max_{t \in [0,T]}\abs{ \zeta^1(t) - \zeta^2(t)} - (\zeta^1(T) - \zeta^2(T))}.
	\end{align*}
\end{theorem}

We remark that a similar result was obtained by Gassiat, Gess, Lions, and Souganidis \cite{GGLS} using slightly different methods, as a tool to study some finer properties of solutions, such as the cancellation of oscillations and speed of propagation.

Both results in Theorem \ref{T:pathstability} follow from the next proposition. The hypotheses require more regularity for the Hamiltonian than is specified by \eqref{A:appHbasic}. The proof of Theorem \ref{T:pathstability} then involves a further regularization of $H$, and the result will follow upon obtaining estimates that do not depend on the regularization parameter.

The proof below uses similar strategies as those in \cite{FGLS, Se, LSbook}.
 
\begin{proposition}\label{P:uvestimate}
	Assume that $H$ satisfies \eqref{A:appHbasic},
	\begin{equation} \label{A:Hregularunifconvex}
		H \in C^2_b(B_R \times \RR^d) \quad \text{for all $R > 0$, and} \quad D^2_p H \quad \text{is strictly positive.}
	\end{equation}
	 For $u_0, v_0 \in UC(\RR^d)$ and $\zeta, \eta \in C^1([0,\oo))$ with $\zeta_0 = \eta_0$, let $u$ be a sub-solution of
	\begin{equation*}
			u_t = H(Du, x)\dot \zeta(t) \quad \text{in } \RR^d \times (0,\oo), \qquad u(\cdot,0) = u_0 \quad \text{on } \RR^d,
	\end{equation*}
	and $v$ a super-solution of
	\begin{equation*}
			v_t = H(Dv, x)\dot \eta(t) \quad \text{in } \RR^d \times (0,\oo), \qquad v(\cdot,0) = v_0 \quad \text{on } \RR^d.
	\end{equation*}
	Then, for all $T > 0$ and $0 < \lambda < \pars{ \max_{0 \le t \le T} (\zeta_t - \eta_t)_-}^{-1}$,
	\begin{align*}
		\sup_{(x,y,t) \in \RR^d \times \RR^d \times [0,T]} &\pars{ u(x,t) - v(y,t) - \pars{ \frac{1}{\lambda} + \zeta_t - \eta_t} \uline{\nu}^* \pars{ \frac{\lambda|x-y|}{1+\lambda (\zeta_t - \eta_t)}} } \\
		&\le \sup_{(x,y) \in \RR^d \times \RR^d} \pars{ u_0(x) - v_0(y) - \frac{1}{\lambda}\oline{\nu}^*(\lambda|x-y|)}.
	\end{align*}
\end{proposition}

Equipped with Proposition \ref{P:uvestimate}, we proceed with the

\begin{proof}[Proof of Theorem \ref{T:pathstability}]
	{\it Step 1.} Assume first that $H$ satisfies \eqref{A:Hregularunifconvex} in addition to \eqref{A:appHbasic}. Applying Proposition \ref{P:uvestimate} to the case $u = v = u^1$ and $\zeta = \eta = \zeta^1$ yields, for all $(x,y,t) \in \RR^d \times \RR^d \times (0,\oo)$, 
	\begin{align*}
		u^1(x,t) - u^1(y,t) &\le \inf_{\lambda > 0} \left\{ \frac{1}{\lambda} \uline{\nu}^*(\lambda|x-y|) + \sup_{s \ge 0 } \left\{ Ls - \frac{1}{\lambda}\oline{\nu}^*(\lambda s) \right\} \right\}\\
		&= \inf_{\lambda > 0} \left\{ \frac{ \uline{\nu}^*(\lambda|x-y|) + \oline{\nu}(L)}{\lambda} \right\}
		= \uline{\nu}^{-1} \pars{ \oline{\nu}(L)} |x-y|.
	\end{align*}
	Thus $\nor{Du^1(\cdot,t)}{\oo} \le \uline{\nu}^{-1} \pars{ \oline{\nu}(L)}$, and similarly for $u^2$.
	
	Now setting $(u,v,\zeta,\eta) := (u^1,u^2,\zeta^1,\zeta^2)$ in Proposition \ref{P:uvestimate} gives
	\[
		u^1(x,t) - u^2(x,t) \le \pars{ \frac{1}{\lambda}  - (\zeta^1_t - \zeta^2_t)} \uline{\nu}^*(0) + \max_{x \in \RR^d} \abs{ u^1_0(x) - u^2_0(x)} + \frac{1}{\lambda} \oline{\nu}(L).
	\]
	The claim follows upon choosing $\lambda = ( \max_{s \in [0,t]} \abs{ \zeta^1_s - \zeta^2_s} )^{-1}$ and using the fact that
	\[
		\uline{\nu}^*(0) = - \min_{r \ge 0} \uline{\nu}(r) = - \uline{\nu}(0) \le \uline{\nu}(0)_-.
	\]
	
	{\it Step 2.} We now return to the general case, where $H$ satisfies only \eqref{A:appHbasic}. Let $\phi \in C^2(\RR^d)$ be nonnegative and supported in $B_1(0)$ with $\int \phi = 1$, and, for $\rho > 0$, define
	\[
		\phi_\rho(z) := \frac{1}{\rho^d} \phi\pars{ \frac{z}{\rho}}
	\]
	and
	\[
		H_\rho(p,x) := \rho |p|^2 + \iint_{\RR^d \times \RR^d} H(q,y)\phi_\rho(p-q) \phi_\rho(x-y)\,dq\,dy.
	\]
	It is straightforward to verify that $\lim_{\rho \to 0} H_\rho = H$ locally uniformly, and $H_\rho$ satisfies both \eqref{A:appHbasic} and \eqref{A:Hregularunifconvex} with the growth functions
	\[
		\oline{\nu}_\rho(s) := \rho s^2 + \oline{\nu}(s + \rho) \quad \text{and} \quad \uline{\nu}_\rho(s) := \rho s^2 + \uline{\nu}\pars{ (s - \rho)_+}.
	\]
	
	Let $u^1_\rho$ and $u^2_\rho$ be as in the statement of Theorem \ref{T:pathstability} for the Hamiltonian $H_\rho$. As proved above, $u^1_\rho$ and $u^2_\rho$ satisfy the Lipschitz bound and stability estimate for $\oline{\nu}_\rho$ and $\uline{\nu}_\rho$. Classical arguments from the theory of viscosity solutions yield the local uniform convergence, as $\rho \to 0$, of $u^j_\rho$ to $u^j$ for $j = 1,2$, where $u^j$ are as in the statement of Theorem \ref{T:pathstability} for the Hamiltonian $H$. Since $\oline{\nu}_\rho$ and $\uline{\nu}_\rho$ converge, as $\rho \to 0$, to $\oline{\nu}$ and $\uline{\nu}$, the proof is complete.
\end{proof}

The rest of this section is devoted to the proof of Proposition \ref{P:uvestimate}. The result is a generalization of Proposition A.2 in \cite{Se}.

For $x,y \in \RR^d$ and $\tau > 0$, define
\[
	\mcl A(x,y,\tau) := \left\{ \gamma \in W^{1,\oo}( [0,\tau], \RR^d) : \gamma_0 = x, \; \gamma_\tau = y \right\}
\]
and
\begin{equation} \label{E:distancefn}
	L(x,y,\tau) := \inf \left\{ \int_0^\tau H^*\pars{ - \dot \gamma_s, {\gamma_s}}\;ds : \gamma \in \mcl A(x,y,\tau) \right\}. 
\end{equation}
We summarize the main properties of this distance function in the next lemma. We omit the proof, as it follows more or less in the same way as in Lemma A.1 of \cite{Se}. 

For $R > 0$, define
\[
	\Delta_R := \left\{ (x,y) \in \RR^d \times \RR^d : |x-y| \le R \right\}.
\]

\begin{lemma}\label{L:convexdistancefunction} 
Assume that $H$ satisfies \eqref{A:Hregularunifconvex}. Then the following hold:
\begin{enumerate}[(a)]
\item \label{L:Lequations} $L$ is a viscosity solution of
	\begin{equation*}
		 \frac{\del L}{\del \tau} = H(D_x L,x)\quad \text{and} \quad \frac{\del L}{\del \tau} = H(-D_y L,y) \quad \text{in } \RR^d \times \RR^d \times (0,\oo).
	\end{equation*}
\item \label{L:Lcoercive} For all $x,y \in \RR^d$ and $\tau > 0$, 
\[
	\tau\oline{\nu}^*\pars{ \frac{|x-y|}{\tau}} \le L(x,y,\tau) \le \tau \uline{\nu}^*\pars{ \frac{|x-y|}{\tau}}.
\]
Furthermore, there exists $\gamma \in \mcl A(x,y,\tau)$ such that $L(x,y,\tau) = \int_0^\tau H^*(- \dot \gamma_s, \gamma_s)\;ds$, and, for some $c \ge 1$ and almost every $s \in [0,\tau]$, 
\[
	 \frac{|x-y|}{c\tau} \le |\dot \gamma_s| \le  \frac{c|x-y|}{\tau}.
\]

\item \label{L:Lbounds} For all $R > 0$, there exists a constant $C = C_R > 0$ such that
	\[
		|D_x L| + |D_y L| \le C \quad \text{and} \quad D^2 L \le C \Id \quad \text{on } \Delta_R \times \left[ \frac{1}{R}, R\right].
	\]
\end{enumerate}
\end{lemma}

The upper bound on $D^2 L$ means that $L$ is semiconcave in space. As the next result demonstrates, this allows $L$ to be used as a test function at an important point in the proof of Proposition \ref{P:uvestimate}, despite the fact that $L$ is not in general $C^1$.

\begin{lemma}\label{L:Ldiff}
	Under the same assumptions as Lemma \ref{L:convexdistancefunction}, assume that $\phi \in C^2(\RR^d \times \RR^d)$ and $L(\cdot,\cdot,\tau_0) - \phi$ attains a local minimum at $(x_0,y_0)$. Then $L$ is differentiable at $(x_0,y_0,\tau_0)$ with
	\[
		\left\{
		\begin{split}
		&(D_xL(x_0,y_0,\tau_0) ,D_yL(x_0,y_0,\tau_0)) = (D_x\phi(x_0,y_0) ,D_y \phi(x_0,y_0))\quad \text{and} \\
		&\frac{\del L}{\del \tau}(x_0,y_0,\tau_0) = H(D_x L(x_0,y_0,\tau_0),x_0) = H(-D_y L(x_0,y_0,\tau_0),y_0).
		\end{split}
		\right.
	\]
\end{lemma}

\begin{proof}
	In view of the semiconcavity of $L(\cdot,\cdot, \tau_0)$ on $\RR^d \times \RR^d$, the super-differential of $L(\cdot,\cdot,\tau_0)$ is nonempty at every point. Meanwhile, $(p_0,q_0) := D\phi(x_0,y_0)$ belongs to the sub-differential of $L(\cdot,\cdot,\tau_0)$ at $(x_0,y_0)$. This implies that $L(\cdot,\cdot,\tau_0)$ is differentiable at $(x_0,y_0)$, and the first line above holds.
	
	Choose $\psi^+,\psi^- \in C^2(\RR^d \times \RR^d)$ such that 
	\[
		\left\{
		\begin{split}
		&\psi^- \le L(\cdot,\cdot,\tau_0) \le \psi^+, \\
		&\psi^-(x_0,y_0) = L(x_0,y_0,\tau_0) = \psi^+(x_0,y_0), \quad \text{and} \\
		&D\psi^-(x_0,y_0) = D\psi^+(x_0,y_0) = (p_0,q_0).
		\end{split}
		\right.
	\]
	The method of characteristics can then be used to construct, for sufficiently small $\mu > 0$, solutions $\Psi^{\pm} \in C^2( \RR^d \times \RR^d \times (\tau_0 - \mu, \tau_0 + \mu))$ of the equations 
	\[
		\frac{\del \Psi^{\pm}}{\del \tau}(x,y,\tau) = H(D_x \Psi^{\pm}(x,y,\tau),x) \quad \text{in } \RR^d \times \RR^d \times (\tau_0 - \mu, \tau_0 + \mu).
	\]
	The comparison principle and Lemma \ref{L:convexdistancefunction}\eqref{L:Lequations} then yield
	\begin{equation}\label{psiineq}
		\Psi^-(x,y,\tau) \le L(x,y,\tau) \le \Psi^+(x,y,\tau) \quad \text{for all } (x,y,\tau) \in \RR^d \times \RR^d \times (\tau_0 - \mu, \tau_0 + \mu).
	\end{equation}
	Finally, the regularity of $H$ and the equations for $\Psi^\pm$ allow for the Taylor expansion
	\begin{align*}
		\Psi^{\pm}(x,y,\tau) &= L(x_0,y_0,\tau_0) + p \cdot(x - x_0) + q \cdot(y - y_0)\\
		& + H(p_0,x_0)(\tau - \tau_0) + O(|x-x_0|^2 + |y-y_0|^2 + |\tau - \tau_0|^2).
	\end{align*}
	Together with \eqref{psiineq}, this shows that $L$ is differentiable at $(x_0,y_0,\tau_0)$ and
	\[
		\frac{\del L}{\del \tau}(x_0,y_0,\tau_0) = H(D_x L(x_0,y_0,\tau_0),x_0).
	\]
	A similar argument using the equation $\frac{d\Psi}{d\tau} = H(-D_y \Psi,y)$ gives the final desired equality
	\[
		\frac{\del L}{\del \tau}(x_0,y_0,\tau_0) = H(-D_y L(x_0,y_0,\tau_0),y_0).
	\]
	\end{proof}
	
\begin{proof}[Proof of Proposition \ref{P:uvestimate}]
	We first note that it suffices to assume that $u_0$ and $v_0$ are bounded. Because the resulting estimates do not depend on $\nor{u_0}{\oo}$ or $\nor{v_0}{\oo}$, the general result can be obtained through an approximation procedure and the local uniform stability of the equations with respect to the initial data.

	Classical viscosity solution arguments show that $z(x,y,t) := u(x,t) - v(y,t)$ is a sub-solution of
	\begin{equation}\label{E:mixeddoubled}
		z_t = H(D_x z,x)\dot \zeta - H(-D_y z, y) \dot \eta \quad \text{in } \RR^d \times \RR^d \times (0,\oo). 
	\end{equation}
	For $0 < \lambda < \pars{ \max_{0 \le t \le T} (\zeta_t - \eta_t)_-}^{-1}$, define
	\[
		\Phi_\lambda(x,y,t) := L\pars{ x,y, \frac{1}{\lambda} + \zeta_t - \eta_t}.
	\]
	A simple computation and Lemma \ref{L:convexdistancefunction}\eqref{L:Lequations} reveal that $\Phi$ satisfies \eqref{E:mixeddoubled} at any point $(x,y,t)$ of differentiability. 
	
	Next, for $0 < \beta < 1$ and $\mu > 0$, define
	\[
		\Psi(x,y,t) := u(x,t) - v(y,t) - \Phi_\lambda(x,y,t) - \frac{\beta}{2}(|x|^2 + |y|^2) - \mu t.
	\]
	The comparison principle from the classical viscosity solution theory yields that $|u(x,t)| \le M$ and $|v(x,t)| \le M$ on $\RR^d \times [0,T]$, where
	\[
		M = \max\left\{ \nor{u_0}{\oo} + \max( |\uline{\nu}(0)|, |\oline{\nu}(0)| ) \max_{0 \le t \le T}|\zeta(t)|, \nor{v_0}{\oo} +  \max( |\uline{\nu}(0)|, |\oline{\nu}(0)| ) \max_{0 \le t \le T} |\eta(t)| \right\}.
	\]
	Therefore, $\Psi$ attains a maximum on $\RR^d \times \RR^d \times [0,T]$ at some $(\hat x, \hat y, \hat t)$ that depends on $\beta$, $\lambda$, and $\mu$. Assume for the sake of contradiction that $\hat t > 0$. 
		
	Rearranging terms in the inequality $\Psi(0,0,\hat t) \le \Psi(\hat x, \hat y, \hat t)$ gives
	\begin{equation}\label{inaball}
		\frac{\beta}{2}(|\hat x|^2 + |\hat y|^2) \le u(\hat x, \hat t) - v(\hat y, \hat t) - (u(\hat 0, \hat t) - v(\hat 0,\hat t)) \le 4M.
	\end{equation}
	The inequality $\Psi(\hat y, \hat y, \hat t) \le \Psi(\hat x, \hat y, \hat t)$ and Lemma \ref{L:convexdistancefunction}\eqref{L:Lcoercive} yield
	\begin{equation}\label{inastrip}
		\pars{ \frac{1}{\lambda} + \zeta_{\hat t} - \eta_{\hat t}} \oline{\nu}^*\pars{ \frac{\lambda|x-y|}{1 + \lambda(\zeta_{\hat t} - \eta_{\hat t}) } }\le u(\hat x, \hat t) - u(\hat y, \hat t) + \frac{\beta}{2} (|\hat y|^2 - |\hat x|^2) \le 6M.
	\end{equation}
	Then \eqref{inaball} and \eqref{inastrip} together imply that, for some $R >0$ depending on $\lambda$, $M$, $\nor{\zeta}{\oo,T}$, and $\nor{\eta}{\oo,T}$, but independent of $\beta$, $(\hat x, \hat y) \in \Omega_{R,\beta}$, where
	\begin{align*}
		\Omega_{R,\beta} &:= \Delta_R \cap B_{R\beta^{-1/2}} 
		= \left\{ (x,y) \in \RR^d \times \RR^d : \pars{ |x|^2 + |y|^2 }^{1/2} \le R \beta^{-1/2} \quad \text{and} \quad |x - y| \le R \right\}.
	\end{align*}
	In the arguments that follow, the constant $C > 0$ depends only on $R$, and may change from line to line.
		
	For $0 < \delta < 1$, set
	\begin{align*}
		\Psi_\delta(x,y,z,w,t) &:= u(x,t) - v(y,t) -\frac{1}{2\delta}(|x-z|^2 + |y-w|^2) - \Phi_\lambda(z,w,t) \\
		&- \frac{\beta}{2}(|z|^2 + |w|^2) - \mu t - \frac{1}{2}\pars{ |x - \hat x|^2 + |y - \hat y|^2 + |t - \hat t|^2}
	\end{align*}
	and assume that the maximum of $\Psi_\delta$ on $\Omega_{R,\beta} \times \Omega_{R,\beta} \times [0,T]$ is attained at $(x_\delta,y_\delta,z_\delta,w_\delta,t_\delta)$. Similar arguments as in the proof of Proposition A.2 from \cite{Se} then yield
	\[
		|x_\delta - z_\delta| + |y_\delta - w_\delta | + |x_\delta - \hat x|^2 + |y_\delta - \hat y|^2 + |t_\delta - \hat t|^2 \le C \delta.
	\]
	Therefore, for sufficiently small $\delta$, $(x_\delta,y_\delta,z_\delta,w_\delta, t_\delta)$ is a local interior maximum point of $\Psi_\delta$ in $\Omega_{R,\beta} \times \Omega_{R,\beta} \times (0,T)$.
				
	Since
	\begin{align*}
		(x,y,t) &\mapsto u(x,t) - v(y,t) - \frac{1}{2\delta} \pars{ |x- z_\delta|^2 + |y-w_\delta|^2}\\
		& - \Phi_\lambda(z_\delta,w_\delta,t) - \mu t - \frac{1}{2} \pars{ |x-\hat x|^2 - |y - \hat y|^2 - |t - \hat t|^2}
	\end{align*}
	attains an interior maximum at $(x_\delta,y_\delta,t_\delta)$, the definition of viscosity solutions yields
	\begin{align*}
		\mu + t_\delta - \hat t + \Phi_{\lambda,t}(z_\delta,w_\delta,t_\delta)
		\le  H\pars{ \frac{x_\delta - z_\delta}{\delta} + x_\delta - \hat x, x_\delta} \dot \zeta_{t_\delta}
		- H\pars{ - \frac{y_\delta - w_\delta}{\delta} - (y_\delta - \hat y), y_\delta} \dot \eta_{t_\delta}.
	\end{align*}
	
	Next, $(z_\delta,w_\delta)$ is a minimum point of
	\begin{align*}
		(z,w) \mapsto \Phi_\lambda(z,w,t_\delta) + \frac{1}{2\delta} (|x_\delta - z|^2 + |y_\delta - w|^2) + \frac{\beta}{2} (|z|^2 + |w|^2).
	\end{align*}
	In view of Lemma \ref{L:Ldiff}, $\Phi_\lambda$ is differentiable at $(z_\delta,w_\delta,t_\delta)$, and so
	\[
		\left\{
		\begin{split}
		&D_x \Phi_\lambda(z_\delta,w_\delta,t_\delta) = \frac{x_\delta - z_\delta}{\delta} - \beta z_\delta,\\
		&D_y \Phi_\lambda(z_\delta,w_\delta,t_\delta) = \frac{y_\delta - w_\delta}{\delta} - \beta w_\delta, \quad \text{and} \\
		&\Phi_{\lambda,t}(z_\delta,w_\delta,t_\delta)
		= H(D_x \Phi_\lambda(z_\delta,w_\delta,t_\delta), z_\delta) \dot \zeta_{t_\delta} - H(-D_y \Phi_\lambda(z_\delta,w_\delta,t_\delta), w_\delta) \dot \eta_{t_\delta}.
		\end{split}
		\right.
	\]
	
	It follows that
	\begin{align*}
		\mu + t_\delta - \hat t + \Phi_{\lambda,t}(z_\delta,w_\delta,t_\delta)
		&\le H\pars{ D_x \Phi_\lambda(z_\delta,w_\delta,t_\delta) + \beta z_\delta + x_\delta - \hat x, x_\delta} \dot \xi_{t_\delta}\\
		&- H\pars{ - D_y \Phi_\lambda(z_\delta, w_\delta,t_\delta) - \beta w_\delta - (y_\delta - \hat y), y_\delta} \dot \zeta_{t_\delta}.
	\end{align*}
	The bounds for $(\hat x, \hat y, \hat t)$ and $(x_\delta,y_\delta,z_\delta,w_\delta,t_\delta)$ and the local Lipschitz regularity of $H$ yield
	\[
		\mu \le C (\beta^{1/2} + \delta^{1/2} + \delta)\pars{\nor{\dot \xi}{\oo,T} + \nor{\dot \zeta}{\oo,T}}.
	\]
	We obtain a contradiction for sufficiently small enough $\delta$ and $\beta$.
	
	Therefore, for all $\mu >0$ and $t \in [0,T]$,
	\begin{align*}
		\lim_{\beta \to 0} &\sup_{(x,y) \in \RR^d \times \RR^d} \pars{ u(x,t) - v(y,t) - \Phi_\lambda(x,y,t) - \frac{\beta}{2} (|x|^2 + |y|^2)}\\
		= &\sup_{(x,y) \in \RR^d \times \RR^d} \pars{ u(x,t) - v(y,t) - \Phi_\lambda(x,y,t)}\\
		\le &\sup_{(x,y) \in \RR^d \times \RR^d} \pars{ u_0(x) - v_0(y) - L(x,y,1/\lambda)} + \mu t.
	\end{align*}
	The desired inequality is established upon letting $\mu \to 0$ and using the bounds in Lemma \ref{L:convexdistancefunction}\eqref{L:Lcoercive}.
\end{proof}
	
\section{Calculation of a nonconvex effective Hamiltonian} \label{S:effective}

Let $F: \RR \to \RR$ be a smooth, even function such that
\begin{equation}\label{A:appnonconvexF}
	\left\{
	\begin{split}
		&\text{for some $0 < \theta_3 < \theta_2 < \theta_1$,}\\
		&F(0) = 0, \; F(\theta_2) = \frac{1}{2}, \; F(\theta_1) = F(\theta_3) = \frac{1}{3}, \; \lim_{r \to \oo} F(r) = +\oo, \\
		&\text{$F$ is strictly increasing on $[0,\theta_2] \cup [\theta_1, +\oo)$ and strictly decreasing on $[\theta_2, \theta_1]$,}
	\end{split}
	\right.
\end{equation}
and, for $0 < s < 1$, define the $1$-periodic function $V_s: \TT \to \TT$ by
\[
	V_s(x) := 
	\begin{dcases}
		\frac{x}{s} & \text{if } 0 \le x \le s \text{ and} \\
		\frac{1-x}{1-s} & \text{if } s < x \le 1.
	\end{dcases}
\]
The goal of this section is to obtain a formula for the effective Hamiltonian associated to
\[
	H_s(p,x) = F(p) - V_s(x).
\]
Some elements of the proof below are used in \cite{LTY}, where it is shown that $\oline{H}_s = \oline{H}_{s'}$ if and only if $s = s'$. For our purposes, in view of Theorem \ref{T:LScriteria}, it is necessary to establish that $\oline{H}_s$ satisfies \eqref{A:Hdiffconv}, which does not follow immediately from standard results from periodic homogenization. The formula in Proposition \ref{P:Hsshape} implies, in particular, that $\oline{H}_s$ is Lipschitz and piecewise smooth, and, hence, \eqref{A:Hdiffconv} is satisfied.

As in \cite{LTY}, define the functions
\[
	\left\{
	\begin{split}
		&\psi_1 := \pars{ F|_{[\theta_1,\oo)} }^{-1} : \left[ \frac{1}{3}, +\oo \right) \to [\theta_1, +\oo), \\
		&\psi_2 := \pars{ F|_{[\theta_2, \theta_1]}}^{-1} : \left[ \frac{1}{3}, \frac{1}{2}\right] \to [\theta_2, \theta_1], \text{ and} \\
		&\psi_3:= \pars{ F|_{[0, \theta_2]}}^{-1} : \left[ 0, \frac{1}{2} \right] \to [0, \theta_2]. 
	\end{split}
	\right.
\]
We identify the following points $0 < p_{+,s} < q_{-,s} < q_+$, between which $\oline{H}_s$ changes its shape:
\begin{equation}\label{flatspotendpoints}
	\left\{
	\begin{split}
		&p_{+,s} := \int_0^{1/3} \psi_3(y)\;dy + \int_{1/2}^1 \psi_1(y)\;dy + \int_{1/3}^{1/2} \left[ s \psi_1(y) + (1-s)\psi_3(y)  \right]dy, \\
		&q_{-,s} := \int_{1/2}^{4/3} \psi_1(y)\;dy + \int_{1/3}^{1/2} \left[ s \psi_1(y) + (1-s)\psi_3(y)  \right]dy, \text{ and}\\
		&q_+ := \int_{1/3}^{4/3} \psi_1(y)\;dy.
	\end{split}
	\right.
\end{equation}

\begin{proposition}\label{P:Hsshape}
	The function $\oline{H}_s$ can be characterized as follows:
	\begin{enumerate}[(a)]
	\item If $0 \le p \le p_{+,s}$, then $\oline{H}_s(p) = 0$.
	
	\item If $p_{+,s} \le p \le q_{-,s}$ then $\oline{H}_s(p)$ is the unique constant $\lambda \in [0,1/3]$ for which
	\[
		p = \int_\lambda^{1/3} \psi_3(y)\;dy + \int_{1/2}^{1+\lambda} \psi_1(y)\;dy + \int_{1/3}^{1/2} \left[ s \psi_1(y) + (1-s) \psi_3(y) \right]dy.
	\]
	
	\item If $q_{-,s} \le p \le q_+$, then $\oline{H}_s(p) = \frac{1}{3}$.
	
	\item If $p \ge q_+$, then $\oline{H}_s(p)$ is the unique constant $\lambda \ge \frac{1}{3}$ for which
		\[
			p = \int_{\lambda}^{1+\lambda} \psi_1(y)\;dy.
		\]
		
	\item If $p < 0$, then $\oline{H}_s(p) = \oline{H}_{1-s}(-p)$.
	\end{enumerate}
\end{proposition}

Obtaining the formula for $\oline{H}_s(p)$ involves constructing viscosity solutions of the equation
\begin{equation}\label{E:Hscellproblem}
	F(w'(y)) - V_s(y) = \lambda \quad \text{in } \RR
\end{equation}
such that $w(x) - px$ is periodic, which is possible only for the unique constant $\lambda = \oline{H}_s(p)$. We make use of the following lemma, whose proof is a consequence of the definition of viscosity solutions:
\begin{lemma}\label{L:visccrit}
	Assume that $F(f(y)) + V_s(y) = \lambda$ at all points $y \in \RR$ at which $f$ is continuous, and, whenever
	\[
		y_0 \in \RR, \quad p_1 := f(y_0^-), \quad \text{and} \quad p_2 := f(y_0^+),
	\]
	then $F(p_1) = F(p_2) = \lambda + V_s(y_0)$ and
	\begin{align*}
		&p_1 < p_2 \; \Rightarrow F(p) \ge \lambda + V_s(y_0) \text{ for } p \in [p_1,p_2], \\
		&p_1 > p_2 \; \Rightarrow F(p) \le \lambda + V_s(y_0) \text{ for } p \in [p_2, p_1].
	\end{align*}
	Then $\{ y \mapsto w(y) := \int_0^y f(x)\;dx\}$ is a viscosity solution of \eqref{E:Hscellproblem}, and
	\[
		\oline{H} \pars{ \int_0^1 f(x)\;dx} = \lambda.
	\]
\end{lemma}

For the rest of the section, we construct correctors using Lemma \ref{L:visccrit} as a blueprint, that is, for each $p \in \RR$, we construct $f$ as in the hypotheses of Lemma \ref{L:visccrit} for the correct constant $\oline{H}_s(p)$.

Define the points $p_{0,s} < p_4 < p_3 < p_2 < p_1 < p_{+,s}$ by
\begin{equation}\label{flatspotgradients}
	\left\{
	\begin{split}
		&p_{0,s} := (2s-1) \int_0^{1/3} \psi_3(y)\;dy + (2s-1) \int_{1/3}^1 \psi_1(y)\;dy,\\
		&p_1 := (2s-1)\int_0^{1/3} \psi_3(y)\;dy + \int_{1/2}^1 \psi_1(y)\;dy + \int_{1/3}^{1/2} \left[ s \psi_1(y) + (1-s)\psi_3(y)  \right]dy, \\
		&p_2 := (2s-1)\int_0^{1/3} \psi_3(y)\;dy + \int_{1/2}^1 \psi_1(y)\;dy + \int_{1/3}^{1/2} \left[ s \psi_1(y) - (1-s)\psi_3(y)  \right]dy, \\
		&p_3 := (2s-1)\int_0^{1/3} \psi_3(y)\;dy + \int_{1/2}^1 \psi_1(y)\;dy + \int_{1/3}^{1/2} \left[ s \psi_1(y) - (1-s)\psi_2(y)  \right]dy, \quad \text{and} \\
		&p_4 := (2s-1) \int_0^{1/3} \psi_3(y)\;dy + \int_{1/2}^1\psi_1(y)\;dy + (2s-1) \int_{1/3}^{1/2} \psi_1(y)\;dy.
	\end{split}
	\right.
\end{equation}
The formula for $\oline{H}_s(p)$ will be established for all $p \ge p_{0,s}$, and the formula for the remaining gradients follows because $p_{0,1-s} = -p_{0,s}$ and $\oline{H}_{1-s}(p) := \oline{H}_s(-p)$.

{\bf Case 1: $p_1 \le p \le p_{+,s}$ and $\lambda = 0$}
\[
	f(x) :=
		\begin{dcases}
			\phi_3(V_s(x)) & \text{if } x \in \pars{0, \frac{s}{3} } \cup \pars{ \frac{1+s}{2}, 1 - \tau(1-s)},\\
			\phi_1(V_s(x)) & \text{if } x \in \pars{\frac{s}{3} , \frac{1+s}{2}},\\
			-\phi_3(V_s(x)) & \text{if } x \in \pars{ 1 - \tau(1-s), 1},
		\end{dcases}
\]
where $\tau \in [0,1/3]$ is given uniquely by
\begin{align*}
	p = (2s-1) \int_0^\tau \psi_3(y)\;dy + \int_\tau^{1/3} \psi_3(y)\;dy + \int_{1/3}^{1/2} \left[ s \psi_1(y) + (1-s) \psi_3(y)  \right]dy + \int_{1/2}^1\psi_1(y)\;dy.
\end{align*}

{\bf Case 2: $p_2 \le p \le p_1$ and $\lambda = 0$}
\[
	f(x) :=
		\begin{dcases}
			\phi_3(V_s(x)) & \text{if } x \in \pars{0, \frac{s}{3} } \cup \pars{ \frac{1+s}{2}, 1 - \tau(1-s)},\\
			\phi_1(V_s(x)) & \text{if } x \in \pars{\frac{s}{3} , \frac{1+s}{2}},\\
			-\phi_3(V_s(x)) & \text{if } x \in \pars{ 1 - \tau(1-s), 1},
		\end{dcases}
\]
where $\tau \in [1/3,1/2]$ is given uniquely by
\begin{align*}
		p &= (2s-1) \int_0^{1/3} \psi_3(y)\;dy + \int_{1/3}^\tau \left[ s \psi_1(y) - (1-s) \psi_3(y) \right]dy \\
		&+ \int_\tau^{1/2} \left[s \psi_1(y) - (1-s) \psi_3(y) \right]dy + \int_{1/2}^1 \psi_1(y)\;dy.
\end{align*}
	
{\bf Case 3: $p_3 \le p \le p_2$ and $\lambda = 0$}
\[
		f(x) :=
		\begin{dcases}
			\psi_3(V_s(x)) & \text{if } x \in \pars{ 0, \frac{s}{3}}, \\
			\psi_1(V_s(x)) & \text{if } x \in \pars{ \frac{s}{3}, \frac{1+s}{2}}, \\
			-\psi_2(V_s(x)) & \text{if } x \in \pars{ \frac{1+s}{2}, 1 - \tau(1-s) }, \\
			-\psi_3(V_s(x)) & \text{if } x \in \pars{ 1 - \tau(1-s), 1},
		\end{dcases}
	\]
where $\tau \in [1/3,1/2]$ is given uniquely by
	\begin{align*}
		p &:= (2s-1) \int_0^{1/3} \psi_3(y)\;dy + \int_{1/3}^\tau \left[ s \psi_1(y) - (1-s) \psi_3(y) \right]dy\\
		& + \int_\tau^{1/2} \left[ s \psi_1(y) - (1-s) \psi_2(y) \right]dy + \int_{1/2}^1 \psi_1(y)\;dy.
	\end{align*}	
	
{\bf Case 4: $p_4 \le p \le p_3$ and $\lambda = 0$}
\[
		f(x) :=
		\begin{dcases}
			\psi_3(V_s(x)) & \text{if } x \in \pars{ 0, \frac{s}{3}}, \\
			\psi_1(V_s(x)) & \text{if } x \in \pars{ \frac{s}{3}, \frac{1+s}{2}}, \\
			-\psi_2(V_s(x)) & \text{if } x \in \pars{ \frac{1+s}{2}, 1 - \tau(1-s) }, \\
			-\psi_1(V_s(x)) & \text{if } x \in \pars{ 1 - \tau(1-s), \frac{2+s}{3} }, \\
			-\psi_3(V_s(x)) & \text{if } x \in \pars{ \frac{2+s}{3}, 1},
		\end{dcases}
	\]
	where $\tau \in [1/3,1/2]$ is given uniquely by
	\[
		p := (2s-1) \int_0^{1/3} \psi_3(y)\;dy + (2s-1) \int_{1/3}^\tau \psi_1(y)\;dy + \int_\tau^{1/2} \left[ s \psi_1(y) - (1-s) \psi_2(y) \right]dy + \int_{1/2}^1 \psi_1(y)\;dy.
	\]

{\bf Case 5: $p_{0,s} \le p \le p_4$ and $\lambda = 0$}
	\[
		f(x) :=
		\begin{dcases}
			\psi_3(V_s(x)) & \text{if } x \in \pars{ 0, \frac{s}{3}}, \\
			\psi_1(V_s(x)) & \text{if } x \in \pars{ \frac{s}{3}, 1- \tau(1-s)}, \\
			-\psi_1(V_s(x)) & \text{if } x \in \pars{ 1 - \tau(1-s), \frac{2+s}{3} }, \\
			-\psi_3(V_s(x)) & \text{if } x \in \pars{ \frac{2+s}{3}, 1},
		\end{dcases}
	\]
	where $\tau \in [1/2,1]$ is given uniquely by
	\[
		p := (2s-1) \int_0^{1/3} \psi_3(y)\;dy + (2s-1) \int_{1/3}^\tau \psi_1(y)\;dy + \int_\tau^1 \psi_1(y)\;dy.
	\]

{\bf Case 6: $p_{+,s} \le p \le q_{-,s}$ and $\lambda \in [0,1/3]$ satisfies
	\[
		p = \int_\lambda^{1/3} \psi_3(y)\;dy + \int_{1/2}^{1+\lambda} \psi_1(y)\;dy + \int_{1/3}^{1/2} \left[ s \psi_1(y) + (1-s) \psi_3(y) \right]dy.
	\]
	}
	\[
		f(y) :=
		\begin{dcases}
			\psi_3(\lambda + V_s(x)) & \text{if } s \in \pars{ 0, (1-3\lambda) \frac{s}{3}} \cup \pars{ \frac{1+s}{2} + \lambda (1-s), 1} \\
			\psi_1(\lambda + V_s(x)) & \text{if } x \in \pars{ (1-3\lambda) \frac{s}{3} , \frac{1+s}{2} + \lambda(1-s)}.	
		\end{dcases}
	\]
	
Before moving on to the next case, we define
\[
	q_1 := \int_{1/2}^{4/3} \psi_1(y)\;dy + \int_{1/3}^{1/2} \left[ s \psi_1(y) + (1-s)\psi_2(y)  \right]dy.
\]
	
{\bf Case 7: $q_{-,s} \le p \le q_1$ and $\lambda = 1/3$}

	There exists a unique $\tau \in [1/3,1/2]$ such that
	\[
		p = \int_{1/3}^\tau \left[ s \psi_1(y) + (1-s) \psi_2(y) \right]dy + \int_\tau^{1/2} \left[ s \psi_1(y) + (1-s) \psi_3(y) \right]dy + \int_{1/2}^{4/3} \psi_1(y)\;dy.
	\]
	Let $\mu \in [(5+s)/6, 1]$ be defined by
	\[
		\tau = \frac{1}{3} + \frac{1-\mu}{1-s} \in \left[ \frac{1}{3}, \frac{1}{2} \right],
	\]
	and define
	\[
		f(x) := 
		\begin{dcases}
			\psi_1(1/3 + V_s(x)) & \text{if } x \in \pars{ 0, \frac{5+s}{6}}, \\
			\psi_3(1/3 + V_s(x)) & \text{if } x \in \pars{ \frac{5+s}{6}, \mu }\\
			\psi_2(1/3 + V_s(x)) & \text{if } x \in \pars{ \mu, 1}.
		\end{dcases}
	\]
	
{\bf Case 8: $q_1 \le p \le q_+$ and $\lambda = 1/3$}

	There exists a unique $\tau \in [1/3,1/2]$ such that
	\[
		p = \int_{1/3}^\tau \left[ s \psi_1(y) + (1-s) \psi_2(y) \right]dy + \int_\tau^{4/3} \psi_1(y)\;dy.
	\]
	Let $\mu \in [ (5+s)/6, 1]$ be defined by 
	\[
		\tau = \frac{1}{3} + \frac{1-\mu}{1-s} \in [1/3,1/2],
	\]
	and define
	\[
		f(x) :=
		\begin{dcases}
			\psi_1(1/3 + V_s(x)) & \text{if } x \in (0,\mu), \\
			\psi_2(1/3 + V_s(x)) & \text{if } x \in (\mu, 1).
		\end{dcases}
	\]

{\bf Case 9: If $p \ge q_{+,s}$ and $\lambda \in [1/3,\oo)$ satisfies
	\[
		p = \int_\lambda^{1+\lambda} \psi_1(y)\;dy,
	\]
	}
	then define
	\[
		f(x) := \psi_1(\lambda + V_s(x)).
	\]
	
\bibliography{HJhomogscaling}{}
\bibliographystyle{acm}

\end{document}